\documentclass[11pt]{article}%
\usepackage{graphicx}
\usepackage{amsmath}
\usepackage{amsfonts}
\usepackage{amssymb}%
\providecommand{\U}[1]{\protect\rule{.1in}{.1in}}
\newtheorem{theorem}{Theorem}[section]

\newtheorem{corollary}[theorem]{Corollary}

\newtheorem{definition}[theorem]{Definition}

\newtheorem{lemma}[theorem]{Lemma}
\newtheorem{notation}[theorem]{Notation}

\newtheorem{proposition}[theorem]{Proposition}
\newtheorem{remark}[theorem]{Remark}

\newenvironment{proof}[1][Proof]{\textbf{#1.} }{\hfill\rule{0.5em}{0.5em}}
{\catcode`\@=11\global\let\AddToReset=\@addtoreset
\AddToReset{equation}{section}

\AddToReset{theorem}{section}

\begin{document}

\title{Isolated initial singularities for the viscous Hamilton-Jacobi equation}
\author{Marie Fran\c{c}oise BIDAUT-VERON\thanks{Laboratoire de Math\'{e}matiques et
Physique Th\'{e}orique, CNRS UMR 7350, Facult\'{e} des Sciences, 37200 Tours
France. E-mail address:veronmf@univ-tours.fr}
\and Nguyen Anh DAO\thanks{Laboratoire de Math\'{e}matiques et Physique
Th\'{e}orique, CNRS UMR 7350, Facult\'{e} des Sciences, 37200 Tours France.
E-mail address:Anh.Nguyen@lmpt.univ-tours.fr}}
\date{}
\maketitle

\begin{abstract}
Here we study the nonnegative solutions of the viscous Hamilton-Jacobi
equation
\[
u_{t}-\Delta u+|\nabla u|^{q}=0
\]
in $Q_{\Omega,T}=\Omega\times\left(  0,T\right)  ,$ where $q>1,T\in\left(
0,\infty\right]  ,$ and $\Omega$ is a smooth bounded domain of $\mathbb{R}%
^{N}$ containing $0,$ or $\Omega=\mathbb{R}^{N}.$ We consider solutions with a
possible singularity at point $(x,t)=(0,0).$ We show that if $q\geq q_{\ast
}=(N+2)/(N+1)$ the singularity is removable. For $1<q<q_{\ast}$, we prove the
uniqueness of a very singular solution without condition as $|x|\rightarrow
\infty$; we also show the existence and uniqueness of a very singular solution
of the Dirichlet problem in $Q_{\Omega,\infty},$ when $\Omega$ is bounded. We
give a complete description of the solutions in each case.\bigskip

\bigskip

\textbf{A.M.S. Subject Classification }35K15, 35K55, 35B33, 35B65, 35D30

\end{abstract}

.\pagebreak

\section{Introduction}

Let $\Omega$ be a smooth bounded domain of $\mathbb{R}^{N}$ containing $0,$ or
$\Omega=\mathbb{R}^{N}$, and $\Omega_{0}=\Omega\backslash\{0\}.$ Here we
consider the nonnegative solutions of the viscous parabolic Hamilton-Jacobi
equation%
\begin{equation}
u_{t}-\Delta u+|\nabla u|^{q}=0 \label{un}%
\end{equation}
in $Q_{\Omega,T}=\Omega\times\left(  0,T\right)  ,$ where $q>1$, with a
possible singularity at point $(x,t)=(0,0),$ in the sense:
\begin{equation}
\lim_{t\rightarrow0}\int_{\Omega}u(.,t)\varphi dx=0,\qquad\forall\varphi\in
C_{c}(\Omega_{0}), \label{R}%
\end{equation}
which means \textit{formally} that $u(x,0)=0$ for $x\neq0.$ \medskip\ 

Such a problem was first considered for the semi-linear equation with a lower
term or order $0:$
\begin{equation}
u_{t}-\Delta u+|u|^{q-1}u=0\hspace{0.5cm}\text{in }Q_{\Omega,T}, \label{bf}%
\end{equation}
with $q>1.$ In a well-known article of Brezis and Friedman \cite{BrFr}, it was
shown that the problem admits a critical value $q_{c}=(N+2)/N$. For any
$q<q_{c},$ and any bounded Radon measure $u_{0}\in\mathcal{M}_{b}(\Omega),$
there exists a unique solution of (\ref{bf}) with Dirichlet conditions on
$\partial\Omega$ with initial data $u_{0},$ in the weak $^{\ast}$ sense:
\begin{equation}
\lim_{t\rightarrow0}\int_{\Omega}u(.,t)\varphi dx=\int_{\Omega}\varphi
du_{0},\qquad\forall\varphi\in C_{c}(\Omega). \label{me}%
\end{equation}
Moreover, from \cite{BrPT} and \cite{KP},there exists a \textit{very singular
solution } in $\mathbb{R}^{N}$, satisfying
\begin{equation}
\lim_{t\rightarrow0}\int_{B_{r}}u(.,t)dx=\infty,\qquad\forall\text{ }%
B_{r}\subset\Omega,
\end{equation}
and it is the limit as $k\rightarrow\infty$ of the solutions with initial data
$k\delta_{0},$ where $\delta_{0}$ is the Dirac mass at $0$\noindent; its
uniqueness, obtained in \cite{Os}, is also a consequence of the general
results of \cite{MaVe}. For any $q\geqq q_{c},$ such solutions do not exist,
and the singularity is \textit{removable}, in other words any solution of
(\ref{bf}), (\ref{R}) satisfies $u\in C^{2,1}\left(  \Omega\times\left[
0,T\right)  \right)  $ and $u(x,0)=0$ in $\Omega,$ see again \cite{BrFr}%
.\medskip

The problem was extended in various directions, where the Laplacian is
replaced by the porous medium operator $\Delta(\left\vert u\right\vert
^{m-1}u)$, see among them \cite{PT}, \cite{KPVa}, \cite{KVe}, \cite{Kw}%
,\cite{Le1}, \cite{Le2}, or the $p$-Laplacian $\Delta_{p}u$, see for example
\cite{KP2}, \cite{PW}, \cite{KVa}. \medskip

Concerning equation (\ref{un}), up to now, the description was not yet
complete. Here another critical value is involved:
\[
q_{\ast}=\frac{N+2}{N+1}.
\]

In the case $\Omega=\mathbb{R}^{N}$, we define a very singular solution
(called VSS) in $Q_{\mathbb{R}^{N},\infty}$ as any function $u\in L_{loc}%
^{1}(Q_{\mathbb{R}^{N},\infty}),$ such that $|\nabla u|\in L_{loc}%
^{q}(Q_{\mathbb{R}^{N},\infty}),$ satisfying equation (\ref{un}) in
$\mathcal{D}^{\prime}(Q_{\mathbb{R}^{N},\infty}),$ and conditions
\begin{equation}
\lim_{t\rightarrow0}\int_{\mathbb{R}^{N}}u(.,t)\varphi dx=0,\qquad
\forall\varphi\in C_{c}(\mathbb{R}^{N}\backslash\left\{  0\right\}  ).
\label{RR}%
\end{equation}%
\begin{equation}
\lim_{t\rightarrow0}\int_{B_{r}}u(.,t)dx=\infty,\qquad\forall r>0. \label{VS}%
\end{equation}
For $q\in\left(  1,q_{\ast}\right)  ,$ it was shown in \cite{BeLa99} that, for
any $u_{0}\in\mathcal{M}_{b}(\mathbb{R}^{N}),$ there exists a solution $u$
with initial data $u_{0},$ unique in a suitable class, which was enlarged in
\cite{BASoWe}. The existence of a radial self-similar VSS $U$ in
$Q_{\mathbb{R}^{N},\infty},$ unique in that class, was obtained in \cite{QW};
independently in \cite{BeLaEx}, proved the existence of a VSS as a limit as
$k\rightarrow\infty$ of the solutions with initial data $k\delta_{0}$. From
\cite{BeKoLa}, it is unique among (possibly nonradial) functions such that
\begin{equation}
\lim_{t\rightarrow0}\int_{\mathbb{R}^{N}\backslash B_{r}}U(.,t)dx=0,\qquad
\forall r>0, \label{vi}%
\end{equation}%
\begin{equation}
U\in C^{2,1}\left(  Q_{\mathbb{R}^{N},\infty}\right)  \cap C((0,\infty
);L^{1}(\mathbb{R}^{N}))\cap L_{loc}^{q}((0,\infty);W^{1,q}(\mathbb{R}^{N})),
\label{plic}%
\end{equation}%
\begin{equation}
\sup_{t>0}(t^{N/2}\left\Vert u(.,t)\right\Vert _{L^{\infty}(\mathbb{R}^{N}%
)}+t^{(q(N+1)-N)/2q}\left\Vert \nabla(u^{(q-1)/q}(.,t))\right\Vert
_{L^{\infty}(\mathbb{R}^{N})}<\infty\label{ploc}%
\end{equation}
If $q\geqq q_{\ast},$ it was proved in \cite{BeLaEx} that there is no solution
$u$ in $Q_{\mathbb{R}^{N},T}$ with initial data $\delta_{0},$ under the
constraints
\begin{equation}
u\in C((0,T);L^{1}(\mathbb{R}^{N}))\cap L^{q}((0,T);W^{1,q}(\mathbb{R}^{N});
\label{res}%
\end{equation}
and the nonexistence of VSS was stated as an open problem. \medskip

In the case of the Dirichlet problem in $Q_{\Omega,T},$ with $\Omega$ bounded,
similar results were obtained in \cite{BeDa}: for $q\in\left(  1,q_{\ast
}\right)  $ and any $u_{0}\in\mathcal{M}_{b}(\Omega),$ there exists a solution
$u$ such that
\begin{equation}
u\in C((0,T);L^{1}(\Omega))\cap L^{1}((0,T);W_{0}^{1,1}(\Omega),\qquad
\left\vert \nabla u\right\vert ^{q}\in L^{1}\left(  Q_{\Omega,T}\right)  ,
\label{ress}%
\end{equation}
satisfying (\ref{me}) for any $\varphi\in C_{b}(\Omega),$ and unique in that
class; for $q\geqq q_{\ast}$ there exists no solution in this class when
$u_{0}$ is a Dirac mass; the existence or nonexistence of a VSS was not
studied.\medskip

In this article we answer to these questions and complete the description of
the solutions. \medskip

In Section \ref{two} we introduce the notion of weak solutions and study their
first properties. We extend some universal estimates of \cite{CLS} for the
Dirichlet problem. When $q\leqq2,$ we show that the solutions are smooth,
improving some results of \cite{BeKoLa}, see Theorems \ref{T.2.1} and
\ref{Dir}. We point out some particular singular solutions or supersolutions,
fundamental in the sequel. We also give some trace results, in the footsteps
of \cite{MaVe}, and apply them to the solutions of (\ref{un}), (\ref{R}%
).\medskip

Our main result is the \textit{removability} in the supercritical case $q\geqq
q_{\ast},$ proved in Section \ref{3}, extending the results of \cite{BrFr} to
equation (\ref{un}).

\begin{theorem}
\label{stop} Assume $q\geqq q_{\ast}.$ Let $\Omega$ be any domain in
$\mathbb{R}^{N}.$ Let $u\in L_{loc}^{1}(Q_{\Omega,T}),$ such that $|\nabla
u|\in L_{loc}^{q}(Q_{\Omega,T}),$ be any solution of problem%
\[
(P_{\Omega})\left\{
\begin{array}
[c]{c}%
u_{t}-\Delta u+|\nabla u|^{q}=0\hspace{0.5cm}\text{in }\mathcal{D}^{\prime
}(Q_{\Omega,T}),\\
\\
\lim_{t\rightarrow0}\int_{\Omega}u(.,t)\varphi dx=0,\qquad\forall\varphi\in
C_{c}(\Omega_{0}),
\end{array}
\right.
\]
Then the singularity is removable, in the following sense:\medskip

If $q\leqq2,$ then $u\in C(\Omega\times\left[  0,T\right)  )$ and
$u(x,0)=0,\;\forall x\in\Omega.$\medskip

If $q>2,$ then $u$ is locally bounded near $0,$ and for any domain
$\omega\subset\subset\Omega,$%
\[
\lim_{t\rightarrow0}(\sup_{Q_{\omega,t}}u)=0.
\]

\end{theorem}

Observe that our conclusions hold \textit{without any condition as}\textbf{
}$\left\vert x\right\vert \rightarrow\infty$ if $\Omega=\mathbb{R}^{N},$ or
near $\partial\Omega$ when $\Omega\neq\mathbb{R}^{N}$. As a consequence, for
$q\geqq q_{\ast},\medskip$

(i) there exists\textbf{ }\textit{no VSS in }$Q_{\mathbb{R}^{N},\infty}$ in
the sense above.\medskip

(ii) there exists\textbf{ }\textit{no solution of}\textbf{ }$(P_{\Omega})$
\textit{with a Dirac mass} at $(0,0)$, without assuming (\ref{res}) or
(\ref{ress})$.\medskip$

We give different proofs of Theorem \ref{stop} according to the values of $q.$
For $q\leqq2$, we take benefit of the regularity of the solutions shown in
Section \ref{two}\noindent. When $q<2,$ we make use of supersolutions, and the
difficult case is the critical one $q=q_{\ast}.$ When $q\geqq2,$ our proof is
based on a change of unknown, and on our trace results; the case $q>2$ is the
most delicate, because of the lack of regularity.$\medskip$

Besides, if $\Omega=\mathbb{R}^{N},$ we can show a \textit{global
removability, without condition at }$\infty$\textit{:}

\begin{theorem}
\label{comp}Under the assumptions of Theorem \ref{stop} with $\Omega
=\mathbb{R}^{N},$ then
\[
u(x,t)\equiv0,\qquad\text{a.e. in }\mathbb{R}^{N}\text{,\quad\ for any }t>0.
\]

\end{theorem}

In Section \ref{vs}, we complete the study of the subcritical case $q<q_{\ast
}$. Our main result in this range is the uniqueness of the VSS in
$Q_{\mathbb{R}^{N},\infty}$ \textit{without any condition}:

\begin{theorem}
\label{unique}Let $q\in\left(  1,q_{\ast}\right)  .$ Then there exists a
unique VSS in $Q_{\mathbb{R}^{N},\infty}.$
\end{theorem}

\noindent Moreover we give a complete description of the solutions:

\begin{theorem}
\label{desc}Let $q\in\left(  1,q_{\ast}\right)  .$ Let $u\in L_{loc}%
^{1}(Q_{\mathbb{R}^{N},\infty}),$ be any function such that $|\nabla u|\in
L_{loc}^{q}(Q_{\mathbb{R}^{N},\infty}),$ solution of equation (\ref{un}) in
$\mathcal{D}^{\prime}(Q_{\mathbb{R}^{N},\infty}),$ and satisfying (\ref{RR}). Then

$\bullet$ either (\ref{VS}) holds and $u=U,$

$\bullet$ or there exists $k>0$ such that $u(.,0)=$ $k\delta_{0}$ in the weak
sense of $\mathcal{M}_{b}(\mathbb{R}^{N}):$
\begin{equation}
\lim_{t\rightarrow0}\int_{\mathbb{R}^{N}}u(.,t)\varphi dx=k\varphi
(0),\qquad\forall\varphi\in C_{b}(\mathbb{R}^{N}), \label{ink}%
\end{equation}
and $u$ is the unique solution satisfying (\ref{ink}),

$\bullet$ or $u\equiv0.$\medskip
\end{theorem}

We also consider the Dirichlet problem in $Q_{\Omega,T}$ when $\Omega$ is
bounded:
\begin{equation}
(D_{\Omega,T})\left\{
\begin{array}
[c]{c}%
u_{t}-\Delta u+|\nabla u|^{q}=0\qquad\text{in }Q_{\Omega,T}\\
u=0\hspace{0.5cm}\text{on }\partial\Omega\times\left(  0,\infty\right)  .
\end{array}
\right.  \label{DP}%
\end{equation}
We give a notion of VSS for this problem, generally nonradial, and show the
parallel of Theorem \ref{unique}:

\begin{theorem}
\label{vss} Assume that $q\in\left(  1,q_{\ast}\right)  $ and $\Omega$ is a
smooth bounded domain of $\mathbb{R}^{N}.$ Then there exists a unique VSS of
problem $(D_{\Omega,\infty}).$
\end{theorem}

\noindent Finally we describe all the solutions as above. \medskip

In conclusion, $q_{\ast}$ clearly appears as the upperbound for existence of
solutions with an isolated singularity at time $0$. We refer to \cite{BiDao}
for the study of equation (\ref{un}) or more general quasilinear parabolic
equations with rough initial data, where we give new decay and uniqueness
properties. The problem of removability of nonpunctual singularities will be
the object of a further article.

\section{Weak solutions and regularity\label{two}}

\subsection{First properties of the weak solutions}

We set $Q_{\Omega,s,\tau}=\Omega\times\left(  s,\tau\right)  ,$ for any domain
$\Omega\subset\mathbb{R}^{N},$ any $-\infty\leqq s<\tau\leqq\infty,$ thus
$Q_{\Omega,T}=Q_{\Omega,0,T}.$

\begin{definition}
For any function $\Phi\in L_{loc}^{1}(Q_{\Omega,T}),$ we say that a function
$U$ is a weak solution (resp. subsolution, resp. supersolution) of equation
\begin{equation}
U_{t}-\Delta U=\Phi\qquad\text{in }Q_{\Omega,T}, \label{uphi}%
\end{equation}
if $U\in L_{loc}^{1}(Q_{\Omega,T})$ and, for any $\varphi\in\mathcal{D}%
^{+}(Q_{\Omega,T}),$
\[
\int_{0}^{T}\int_{\Omega}(U\varphi_{t}+U\Delta\varphi+\Phi\varphi
)dxdt=0\qquad(\text{resp.}\leqq,\text{ resp.}\geqq).
\]

\end{definition}

In all the sequel we use regularization arguments by to deal with weak solutions:

\begin{notation}
For any function $u$ $\in L_{loc}^{1}(Q_{\Omega,T}),$ we set
\[
u_{\varepsilon}=u\ast\varrho_{\varepsilon},
\]
where $(\varrho_{\varepsilon})$ is sequence of mollifiers in $\left(
x,t\right)  \in\mathbb{R}^{N+1}.$ Then $u_{\varepsilon}$ is well defined in
$Q_{\Omega,s,\tau}$ for any domain $\omega\subset\subset\Omega$ and
$0<s<\tau<T$ and $\varepsilon>0$ small enough.
\end{notation}

\begin{lemma}
\label{appro} Any solution (resp. subsolution) $U$ of (\ref{uphi}) such that
$U\in C((0,T);L_{loc}^{1}(\Omega))$ satisfies also for any nonnegative
$\varphi\in C_{c}^{\infty}(\Omega\times\left[  0,T\right]  )$ and any
$s,\tau\in(0,T),$
\begin{equation}
\int_{\Omega}U(.,\tau)\varphi(.,\tau)dx-\int_{\Omega}U(.,t)\varphi
(.,t)dx-\int_{s}^{\tau}\int_{\Omega}(U\varphi_{t}+U\Delta\varphi+\Phi
\varphi)dxdt=0\qquad(\text{resp.}\leqq0) \label{ger}%
\end{equation}
and for any nonnegative $\psi\in$ $C_{c}^{2}\left(  \Omega\right)  ,$%
\begin{equation}
\int_{\Omega}U(.,\tau)\psi dx-\int_{\Omega}U(.,s)\psi dx-\int_{s}^{\tau}%
\int_{\Omega}(U\Delta\psi+\Phi\psi)dxdt=0\qquad(\text{resp.}\leqq0).
\label{gir}%
\end{equation}

\end{lemma}

\begin{proof}
The regularization gives the equation $(U_{\varepsilon})_{t}-\Delta
U_{\varepsilon}=\Phi_{\varepsilon},$ and the relations (\ref{ger}),
(\ref{gir}) hold for $U_{\varepsilon},\Phi_{\varepsilon},$ and for $U,\Phi$ as
$\varepsilon\rightarrow0.$ Indeed, $\int_{\Omega}U_{\varepsilon}%
(.,\tau)\varphi(.,\tau)dx$ converges to $\int_{\Omega}U(.,\tau)\varphi
(.,\tau)dx$ for almost any $\tau,$ see for example \cite{AS}, hence the
relations hold for any $s,\tau$ by continuity.\medskip
\end{proof}

Next we make precise our notion of solution of equation (\ref{un}).

\begin{definition}
(i) We say that a nonnegative function $u$ is a weak solution of equation
(\ref{un}) in $Q_{\Omega,T},$ if $u\in L_{loc}^{1}(Q_{\Omega,T}),|\nabla
u|^{q}\in L_{loc}^{1}(Q_{\Omega,T}),$ and $u$ is a weak solution of the
equation in the sense above:
\[
\int_{0}^{T}\int_{\Omega}(-u\varphi_{t}-u\Delta\varphi+|\nabla u|^{q}%
\varphi)dxdt=0,\qquad\forall\varphi\in\mathcal{D}(Q_{\Omega,T}).
\]

\noindent(ii) We say that $u$ is a weak solution of the Dirichlet problem
$(D_{\Omega,T})$ if it is a weak solution of (\ref{un}) in $Q_{\Omega,T},$
such that
\[
u\in L_{loc}^{1}((0,T);W_{0}^{1,1}(\Omega))\cap C((0,T);L^{1}(\Omega
)),\quad\text{and }|\nabla u|\in L_{loc}^{q}((0,T);L^{q}(\Omega)).
\]

\end{definition}

We first observe that the regularization keeps the subsolutions, which allow
to give local estimates:

\begin{lemma}
\label{subso}Let $u$ be a weak nonnegative subsolution of (\ref{un}) in
$Q_{\Omega,T}.$ Let $\omega$ be any domain $\omega\subset\subset\Omega$ and
$0<s<\tau<T$. Then for $\varepsilon$ small enough, $u_{\varepsilon}$ is a
subsolution of equation (\ref{un}) in $Q_{\omega,s,\tau}.$
\end{lemma}

\begin{proof}
The function $u_{\varepsilon}$ satisfies
\[
(u_{\varepsilon})_{t}-\Delta u_{\varepsilon}+\left\vert \nabla u\right\vert
^{q}\ast\varrho_{\varepsilon}\leqq0,
\]
in $\mathcal{D}^{\prime}(Q_{\omega,s,\tau})$ for $\varepsilon$ small enough.
We find easily that
\begin{equation}
|\nabla u_{\varepsilon}|^{q}\leqq\left\vert \nabla u\right\vert ^{q}%
\ast\varrho_{\varepsilon}\qquad\text{in }Q_{\omega,s,\tau}, \label{kine}%
\end{equation}
from the H\"{o}lder inequality, since $\varrho_{\varepsilon}$ has a mass 1;
thus $|\nabla u_{\varepsilon}|^{q}\in L_{loc}^{1}(Q_{\omega,s,\tau})$ and
\begin{equation}
(u_{\varepsilon})_{t}-\Delta u_{\varepsilon}+|\nabla u_{\varepsilon}|^{q}%
\leqq0. \label{epsi}%
\end{equation}

\end{proof}

Next we recall some well known properties:

\begin{lemma}
\label{subcal}Any weak nonnegative solution of equation (\ref{un}) satisfies%
\begin{equation}
u\in L_{loc}^{\infty}(Q_{\Omega,T}),\qquad\nabla u\in L_{loc}^{2}(Q_{\Omega
,T}),\qquad u\in C((0,T);L_{loc}^{r}(\Omega)),\quad\forall r\geqq1.
\label{2.1}%
\end{equation}
As a consequence, it satisfies\medskip

\noindent(i) for any $\varphi\in C_{c}^{1}(Q_{\Omega,T}),$
\begin{equation}
\int_{0}^{T}\int_{\Omega}(-u\varphi_{t}+\nabla u.\nabla\varphi+|\nabla
u|^{q}\varphi)dxdt=0, \label{fur}%
\end{equation}
(ii) for any $s,\tau\in(0,T),$ and any $\varphi\in C^{1}((0,T);C_{c}%
^{1}\left(  \Omega\right)  ),$
\begin{equation}
\int_{\Omega}u(.,\tau)\varphi(.,\tau)dx-\int_{\Omega}u(.,s)\varphi
(.,s)dx+\int_{s}^{\tau}\int_{\Omega}(-u\varphi_{t}+\nabla u.\nabla
\varphi+|\nabla u|^{q}\varphi)dxdt=0 \label{ffr}%
\end{equation}
(iii) for any $s,\tau\in(0,T),$ and any $\psi\in$ $C_{c}^{1}\left(
\Omega\right)  ,$
\begin{equation}
\int_{\Omega}u(.,\tau)\psi dx-\int_{\Omega}u(.,s)\psi dx+\int_{s}^{\tau}%
\int_{\Omega}(\nabla u.\nabla\psi+|\nabla u|^{q}\psi)dxdt=0 \label{fgr}%
\end{equation}

\end{lemma}

\begin{proof}
The function $u\in L_{loc}^{1}(Q_{\Omega,T})$ is nonnegative and subcaloric,
then regularizing $u$ by $u_{\varepsilon},$ we get $u\in L_{loc}^{\infty
}(Q_{\Omega,T}),$ see for example \cite{BrFr}. Otherwise for any domains
$\omega\subset\subset\omega^{\prime}\subset\subset\Omega,$ taking $\psi\in$
$C_{c}^{1}\left(  \Omega\right)  $ with support in $\omega^{\prime}$ such that
$\psi\equiv1$ on $\omega$, $\psi\left(  \Omega\right)  \subset\left[
0,1\right]  ,$ we find
\begin{align*}
&  \int_{\Omega}u_{\varepsilon}^{2}(.,\tau)\psi^{2}dx-\int_{\Omega
}u_{\varepsilon}^{2}(.,s)\psi^{2}dx+\int_{s}^{\tau}\int_{\Omega}\left\vert
\nabla u_{\varepsilon}\right\vert ^{2}\psi^{2}dx\\
&  \leqq2\int_{s}^{\tau}\int_{\Omega}u_{\varepsilon}\left\vert \nabla
u_{\varepsilon}\right\vert \left\vert \nabla\psi\right\vert dx\leqq\frac{1}%
{2}\int_{s}^{\tau}\int_{\Omega}\left\vert \nabla u_{\varepsilon}\right\vert
^{2}\psi^{2}dx+4\int_{s}^{\tau}\int_{\Omega}u_{\varepsilon}^{2}\left\vert
\nabla\psi\right\vert ^{2}dx;
\end{align*}
hence $\nabla u\in L_{loc}^{2}(Q_{\Omega,T})$ from the Fatou Lemma, and
\begin{equation}
\left\Vert \nabla u\right\Vert _{L^{2}(Q_{\omega,s,\tau})}\leqq C(\left\Vert
u(.,s)\right\Vert _{L^{2}(Q_{\omega^{\prime},s,\tau})}+\left\Vert u\right\Vert
_{L^{2}(Q_{\omega^{\prime},s,\tau})})\leqq C\left\Vert u\right\Vert
_{L^{\infty}(Q_{\omega^{\prime},s,\tau})}, \label{pre}%
\end{equation}
with $C=C(N,\omega,\omega^{\prime})$. Then (\ref{fur}) holds for any
$\varphi\in\mathcal{D}(Q_{\Omega,T}).$ Moreover, since $|\nabla u|^{q}\in
L_{loc}^{1}(Q_{\Omega,T}),$ the function $u$ lies in the set
\begin{equation}
E=\left\{  v\in L_{loc}^{2}((0,T);W_{loc}^{1,2}(\Omega)):v_{t}\in L_{loc}%
^{2}((0,T);W^{-1,2}(\Omega))+L_{loc}^{1}\left(  Q_{\Omega,T}\right)  \right\}
\label{eu}%
\end{equation}
From a local version of \cite[Theorem 1.1]{Po}, we have $E\subset
C((0,T);L_{loc}^{1}(\Omega)).$ Then (\ref{ffr}) and (\ref{fgr}) follow.
Moreover $u\in L_{loc}^{\infty}(Q_{\Omega,T}),$ then $u\in C((0,T);L_{loc}%
^{r}(\Omega))$ for any $r>1.$\medskip
\end{proof}

In the case of the Dirichlet problem $(D_{\Omega,T}),$ the regularization does
not provide estimates up to the boundary, thus we use another argument: the
notion of \textit{entropy solution} that we recall now. For any $k>0$ and
$r\in\mathbb{R},$ we define as usual $T_{k}(r)=\max(-k,\min(k,r))$ the
truncation function, and $\Theta_{k}(r)=\int_{0}^{r}T_{k}(s)ds.$

\begin{definition}
Let $s<\tau,$ and $f\in L^{1}(Q_{\Omega,s,\tau})$ and $u_{s}\in L^{1}%
(\Omega).$ A function $u$ is an entropy solution of the problem%
\begin{equation}
\left\{
\begin{array}
[c]{ccc}%
u_{t}-\Delta u & =f & \text{in }Q_{\Omega,s,\tau},\\
u & =0 & \text{on }\left(  s,\tau\right)  \times\partial\Omega,\\
u(.,s) & =u_{s} & \text{in }\Omega,
\end{array}
\right.  \label{fss}%
\end{equation}
if $u\in C(\left[  s,\tau\right]  ;L^{1}(\Omega)),$ and $T_{k}(u)\in
L^{2}(\left(  s,\tau\right)  ;W_{0}^{1,2}(\Omega))$ for any $k>0,$ and%
\begin{align*}
&  \int_{\Omega}\Theta_{k}(u-\varphi)(.,\tau)dx+\int_{s}^{\tau}\langle
\varphi_{t},T_{k}(u-\varphi)\rangle dt+\int_{s}^{\tau}\int_{\Omega}\nabla
u.\nabla T_{k}(u-\varphi)dxdt\\
&  =\int_{\Omega}\Theta_{k}(u_{s}-\varphi(.,0))dx+\int_{s}^{\tau}\int_{\Omega
}fT_{k}(u-\varphi)dxdt
\end{align*}
for any $\varphi\in L^{2}((s,\tau);W^{1,2}(\Omega))\cap L^{\infty}\left(
Q_{\Omega,\tau}\right)  $ such that $\varphi_{t}\in L^{2}((s,\tau
);W^{-1,2}(\Omega))$.\medskip
\end{definition}

As a consequence, we identify three ways of defining solutions:

\begin{lemma}
\label{ent}Let $0\leqq s<\tau\leqq T,$ and $f\in L^{1}(Q_{\Omega,s,\tau})$ and
$u\in C(\left[  s,\tau\right)  ;L^{1}\left(  \Omega\right)  ),$ $u_{s}=u(s).$
Denoting by $e^{t\Delta}$ the semi-group of the heat equation with Dirichlet
conditions acting on $L^{1}\left(  \Omega\right)  ,$ the three properties are
equivalent:\medskip

\noindent(i) $u\in L_{loc}^{1}((s,\tau);W_{0}^{1,1}\left(  \Omega\right)  )$
and $u_{t}-\Delta u=f,\qquad$in$\hspace{0.05in}\mathcal{D}^{\prime}%
(Q_{\Omega,s,\tau}),$\medskip

\noindent(ii) $u$ is an entropy solution of problem (\ref{fss}) in
$Q_{\Omega,s,\tau},$

\noindent(iii)
\[
u(.,t)=e^{(t-s)\Delta}u_{s}+\int_{s}^{t}e^{(t-\sigma)\Delta}f(\sigma
)d\sigma\qquad\text{in }L^{1}\left(  \Omega\right)  ,\quad\forall t\in\left[
s,\tau\right]  .
\]

\end{lemma}

\begin{proof}
It follows from the existence and uniqueness of the solutions of (i) from
\cite[Lemma 3.4]{BaPi}, as noticed in \cite{BeDa}, and of the entropy
solutions, see \cite{AMST}, \cite{Pri}.\medskip
\end{proof}

We deduce properties of all the \textit{bounded} solutions $u$ of
$(D_{\Omega,T}):$

\begin{lemma}
\label{subdir}Any nonnegative weak solution of problem $(D_{\Omega,T})$, such
that $u\in L_{loc}^{\infty}((0,T);L^{\infty}\left(  \Omega\right)  )$
satisfies $\nabla u\in L_{loc}^{2}(0,T);L^{2}\left(  \Omega\right)  )$ and
$u\in C((0,T);L^{r}(\Omega))$ for any $r\geqq1.$
\end{lemma}

\begin{proof}
Since $u\in C((0,T);L^{1}(\Omega)),$ for any $0<s<\tau<T,$ $u$ is an entropy
solution on $\left[  s,\tau\right]  $ from Lemma \ref{ent}. Since $u$ is
bounded, it follows that $u=T_{k}(u)\in L^{2}(\left(  s,\tau\right)
;W_{0}^{1,2}(\Omega)),$ and
\[
\int_{\Omega}u^{2}(.,\tau)dx-\int_{\Omega}u^{2}(.,s)dx+\int_{s}^{\tau}%
\int_{\Omega}\left\vert \nabla u\right\vert ^{2}dx+\int_{s}^{\tau}\int
_{\Omega}u|\nabla u|^{q}dxdt=0;
\]
and $u\in C((0,T);L^{r}(\Omega))$ as in Lemma \ref{subcal}.
\end{proof}

\subsection{Estimates of the classical solutions of the Dirichlet problem}

First recall some results on the Dirichlet problem in a bounded domain
$\Omega$ with regular initial and boundary data%

\begin{equation}
\left\{
\begin{array}
[c]{l}%
u_{t}-\Delta u+|\nabla u|^{q}=0,\quad\text{in}\hspace{0.05in}Q_{\Omega,T},\\
u=\varphi,\quad\text{on}\hspace{0.05in}\partial\Omega\times(0,T),\\
u(x,0)=u_{0}\geqq0.
\end{array}
\right.  \label{pqr}%
\end{equation}
If $\varphi\equiv0$ and $u_{0}\in C_{0}^{1}\left(  \overline{\Omega}\right)
$, it is well known that problem (\ref{pqr}) admits a unique solution $u\in
C^{2,1}\left(  Q_{\Omega,\infty}\right)  \cap$ $C\left(  \overline{\Omega
}\times\left[  0,\infty\right)  \right)  $ such that $|\nabla u|\in C\left(
\overline{\Omega}\times\left[  0,\infty\right)  \right)  .$ For general
$\varphi\in C(\partial\Omega\times\left[  0,T\right]  ),$ the same happens on
$\left[  0,T\right)  $ if $u_{0}\in C^{1}(\overline{\Omega}),$ and
$u_{0}(x)=\varphi(x,0)$ on $\partial\Omega$. If one only assumes $u_{0}\in
C(\overline{\Omega}),$ there exist a unique solution $u\in C(\overline{\Omega
}\times\left[  0,T\right]  )$ in the viscosity sense, see \cite{BaDLi}, but
$\left\vert \nabla u\right\vert $ may have a blow-up near $\partial\Omega$
when $q>2$.\medskip

Some fundamental universal estimates have been obtained in \cite{CLS}:

\begin{theorem}
[\cite{CLS}]\label{apri}Let $\Omega$ be any smooth bounded domain. Let $q>1,$
and $u_{0}\in C_{0}\left(  \overline{\Omega}\right)  $ be Lipschitz
continuous. Let $u$ be the classical solution of (\ref{pqr}) with $\varphi=0.$
Then there exist functions $B,D\in C((0,\infty))$ depending only of
$N,q,\Omega$, such that such that, for any $t\in(0,T),$%
\begin{equation}
\Vert u(.,t)\Vert_{L^{\infty}(\Omega)}\leqq B(t)d(x,\partial\Omega),
\label{wi}%
\end{equation}%
\begin{equation}
\Vert\nabla u(.,t)\Vert_{L^{\infty}(\Omega)}\leqq D(t). \label{univ}%
\end{equation}

\end{theorem}

In the following Lemma, we extend and make precise estimate (\ref{wi}), with
nonzero data on the lateral boundary$:$

\begin{lemma}
\label{mono}Let $\Omega$ be any smooth bounded domain. Let $q>1.$ Let $u\in
C(\overline{\Omega}\times(0,T))\cap C^{2,1}(Q_{\Omega,T})$ be a nonnegative
solution of equation (\ref{un}) in $Q_{\Omega,T},$ bounded on $\partial
\Omega\times(0,T).$ Then there is a constant $C=C(N,q,\Omega)$ such that for
any $\forall t\in(0,T),$
\begin{equation}
\Vert u(.,t)\Vert_{L^{\infty}(\Omega)}\leqq C(1+t^{-\frac{1}{q-1}%
})d(x,\partial\Omega)+\sup_{\partial\Omega\times(0,T)}u,. \label{wa}%
\end{equation}

\end{lemma}

\begin{proof}
Let $M=\sup_{\partial\Omega\times(0,T)}u.$ We set $u_{\delta}=u-(M+\delta)$
for any $\delta>0.$ On $\partial\Omega\times(0,T),$ we have $u_{\delta,k}%
\leqq-\delta<0.$ Since $u_{\delta}(0)$ is continuous, there exists
$\Omega_{\delta}\subset\subset\Omega$ such that $u_{\delta}(0)\leqq-\delta/2$
on $\Omega\backslash\Omega_{\delta}$. Then there exists a contant $C_{\delta}$
such that $u_{\delta}(0)\leqq C_{\delta}d(x,\partial\Omega).$ From \cite{CLS},
for any $z\in\partial\Omega,$ there exists a function $b_{z}(x)$ such that,
for some $k,K,A>0$ depending on $\Omega,$ and for any $x\in\Omega,$
\[
kd(x,\partial\Omega)\leqq\inf_{z\in\partial\Omega}b_{z}(x)\leqq Kd(x,\partial
\Omega),\quad b_{z}(x)\leqq A,\quad k\leqq\left\vert \nabla b_{z}%
(x)\right\vert \leqq1,\quad\left\vert \Delta b_{z}(x)\right\vert \leqq K.
\]
Then for any $z\in\partial\Omega,$ there exists a function $w_{z}$ of the form
$w_{z}(x,t)=J(t)b_{z}(x)$ such that $w_{z}$ is a supersolution of equation
(\ref{un}), $w_{z}\geqq0$ on $\partial\Omega,$ and
\[
\lim_{t\rightarrow0}d(x,\partial\Omega)^{-1}w_{z}(x,t)=\infty
\]
uniformly in $\Omega.$ Otherwise $J$ can be chosen explicitly by
$J(t)=C($Arc$\tan t)^{-1/(q-1)}$ with $C^{q-1}=k^{-q}(K\pi/2+A/(q-1))$. Thus
there exists $\tau_{\delta}>0$ such that $w_{z}(x,\tau)\geqq u_{\delta}(0)$
for $\tau\leqq\tau_{\delta}.$ Since $u_{\delta}$ is a solution of (\ref{un}),
the function $w_{z}(x,\tau+t)-u_{\delta}(x,t)$ is nonnegative from the
comparison principle. Letting $\tau\rightarrow0,$ and then $\delta
\rightarrow0$ and finally taking the infimum over $z\in\partial\Omega$ leads
to the estimate
\begin{equation}
u(x,t)\leqq M+KJ(t)d(x,\partial\Omega), \label{we}%
\end{equation}
hence (\ref{wa}) follows with another constant $C>0.$
\end{proof}

\subsection{Regularity for $q\leqq2$}

First of all, we give a result of regularity $\mathcal{C}^{2,1}$\textit{ for
any} \textit{weak} \textit{solution} of equation (\ref{un}) and for any
$q\leqq2$. Such a regularity was obtained in \cite[Proposition 3.2]{BeKoLa}
for the VSS when $q<q_{\ast}$, and the proof was valid up to $q=(N+4)/(N+2).$
We did not find a good reference in the literature under our weak assumptions,
even if a priori estimates can be found in \cite{LSU}, and H\"{o}lderian
properties in \cite{AS}, \cite{Tr}. Our proof is based on a bootstrap
technique, starting from the fact that $u$ is subcaloric. \medskip

We set $\mathcal{W}^{2,1,\rho}(Q_{\omega,s,\tau})=\{u\in L^{\rho}%
(Q_{\omega,s,\tau}):u_{t},\nabla u,D^{2}u\in L^{\rho}(Q_{\omega,s,\tau})\},$
for any $0\leqq s<\tau<T$ and $1\leqq\rho\leqq\infty.$ This space is endowed
with its usual norm.

\begin{theorem}
\label{T.2.1}Let $1<q\leqq2$. Let $\Omega$ be any domain in $\mathbb{R}^{N}$.
Suppose that $u$ is a weak nonnegative solution of (\ref{un}) in $Q_{\Omega
,T}$. \medskip

\noindent(i) Then $u\in\mathcal{C}^{2,1}(Q_{\Omega,T}),$ and there exists
$\gamma\in\left(  0,1\right)  $ such that for any smooth domains
$\omega\subset\subset\omega^{\prime}\subset\subset\Omega,$ and $0<s<\tau<T$
\begin{equation}
\left\Vert u\right\Vert _{C^{2+\gamma,1+\gamma/2}(Q_{\omega,s,\tau})}\leqq
C\Phi(\left\Vert u\right\Vert _{L^{\infty}(Q_{\omega^{\prime},s/2,\tau})}),
\label{gff}%
\end{equation}
where $\Phi$ is a continuous increasing function and $C=C(N,q,\omega
,\omega^{\prime},s,\tau).$\medskip

\noindent(ii) As a consequence, for any sequence $(u_{n})$ of weak solutions
of equation (\ref{un}) in $Q_{\Omega,T},$ uniformly locally bounded, one can
extract a subsequence converging in $C_{loc}^{2,1}(Q_{\Omega,T})$ to a weak
solution $u$ of (\ref{un}) in $Q_{\Omega,T}$.
\end{theorem}

\begin{proof}
(i) $\bullet$ Case $q<2.$ We can write (\ref{2.1}) under the form
\[
u_{t}-\Delta u=f,\qquad f=-|\nabla u|^{q},
\]
and $f\in L_{loc}^{q_{1}}(Q_{\Omega,T})$, with $q_{1}=2/q\in(1,2).$ From
(\ref{2.1}), there holds $u,\nabla u,f\in L_{loc}^{q_{1}}(Q_{\Omega,T}).$ Then
$u\in\mathcal{W}_{loc}^{2,1,q_{1}}(Q_{\Omega,T}),$ see \cite[ theorem
IV.$9.1$]{LSU}. Choosing $\omega^{\prime\prime}$ such that $\omega
\subset\subset\omega^{\prime\prime}\subset\subset\omega^{\prime}$and denoting
$Q=Q_{\omega,s,\tau},$ $Q^{\prime}=Q_{\omega^{\prime},s/2,\tau},$
$Q^{\prime\prime}=Q_{\omega^{\prime\prime},3s/4,\tau},$ we deduce from
(\ref{pre}) that%
\begin{align*}
\left\Vert u\right\Vert _{\mathcal{W}^{2,1,q_{1}}(Q)}  &  \leqq C(\left\Vert
f\right\Vert _{L^{q_{1}}(Q^{\prime\prime})}+\left\Vert u\right\Vert
_{L^{q_{1}}(Q^{\prime\prime})})\leqq C(\left\Vert \nabla u\right\Vert
_{L^{2}(Q^{\prime\prime})}^{q}+\left\Vert u\right\Vert _{L^{\infty}(Q^{\prime
})})\\
&  \leqq C(\left\Vert u\right\Vert _{L^{\infty}(Q^{\prime})}^{q}+\left\Vert
u\right\Vert _{L^{\infty}(Q^{\prime})}),
\end{align*}
with $C=C(N,q,\omega,\omega^{\prime},s,\tau).$ From the Gagliardo-Nirenberg
inequality, there exists $c=c(N,q,\omega)>0$ such that for almost any
$t\in(0,T)$,
\[
\Vert\nabla u(.,t)\Vert_{L^{2q_{1}}(\omega)}\leqq c\Vert u(t)\Vert
_{W^{2,q_{1}}(\omega)}^{1/2}\Vert u(t)\Vert_{L^{\infty}(\omega)}^{1/2}.
\]
Then by integration, $\left\vert \nabla u\right\vert \in L_{loc}^{2q_{1}}(Q),$
and
\begin{equation}
\Vert\nabla u\Vert_{L^{2q_{1}}(Q)}\leqq c\Vert u(t)\Vert_{W^{2,q_{1}}%
(Q)}^{1/2}\Vert u\Vert_{L^{\infty}(Q)}^{1/2}\leqq C_{1}\Phi_{1}(\left\Vert
u\right\Vert _{L^{\infty}(Q^{\prime})})), \label{2.4}%
\end{equation}
with a new constant $C_{1}$ as above, where $\Phi_{1}$ is a continuous
increasing function. Thus $f\in L_{loc}^{q_{2}}(Q_{\Omega,T}),$ with
$q_{2}=(2/q)^{2}\in\left(  q_{1},2q_{1}\right)  $ and $u,\nabla u,f\in
L_{loc}^{q_{2}}(Q_{\Omega,T}),$ in turn $u\in\mathcal{W}_{loc}^{2,1,q_{2}%
}(Q_{\Omega,T}).$ By induction we find that $u\in\mathcal{W}_{loc}^{2,1,q_{k}%
}(\Omega\times(0,T)),$ with $q_{k}=q_{1}^{k},$ for any $k\geqq1,$ and%
\[
\Vert\nabla u\Vert_{L^{2q_{k}}(Q)}\leqq C_{k}\Phi_{k}(\left\Vert u\right\Vert
_{L^{\infty}(Q^{\prime})})
\]
with $C_{k},\Phi_{k}$ as above. Choosing any $k$ so that $q_{k}>N+2,$ we
deduce that $|\nabla u|\in C^{\gamma,\gamma/2}(\omega\times(s,\tau))$ for any
$\gamma\in(0,1),$ see \cite[Lemma II.3.3]{LSU}. Then $f$ is locally
H\"{o}lderian, thus $u\in C^{2+\gamma,1+\gamma/2}(Q_{\omega,s,\tau})$, and
(\ref{gff}) holds. \medskip

$\bullet$ Case $q=2$. We define $Q$ and $Q^{\prime}$ as above, and regularize
by $u_{\varepsilon}$ in $Q^{\prime}$ for $\varepsilon$ small enough$.$ Since
$u$ is locally bounded, $u_{\varepsilon}$ converges to $u$ in $L^{s}%
(Q^{\prime})$ for any $s\geqq1,$ and by extraction $a.e.$ in $Q.$ And
$u_{\varepsilon}$ satisfies the equation in $Q^{\prime}$
\[
(u_{\varepsilon})_{t}-\Delta u_{\varepsilon}+\left\vert \nabla u\right\vert
^{2}\ast\varrho_{\varepsilon}=0.
\]
Defining the functions $z=1-e^{-u}$ in $Q_{\Omega,T},$ and $z^{\varepsilon
}=1-e^{-u_{\varepsilon}}$ in $Q^{\prime},$ we obtain that
\[
(z^{\varepsilon})_{t}-\Delta(z^{\varepsilon})+h_{\varepsilon}=0,
\]
where $h_{\varepsilon}=e^{-u_{\varepsilon}}\left(  \left\vert \nabla
u\right\vert ^{2}\ast\varrho_{\varepsilon}-|\nabla u_{\varepsilon}%
|^{2}\right)  \geqq0$ from (\ref{kine}). Then $\left\vert \nabla u\right\vert
^{2}\ast\varrho_{\varepsilon}$ converges to $|\nabla u|^{2}$ and $|\nabla
u_{\varepsilon}|^{2}$ converges to $|\nabla u|^{2}$ in $L_{loc}^{1}%
(Q_{\Omega,T}),$ thus $h_{\varepsilon}$ tends to $0$ in $L_{loc}^{1}%
(Q_{\Omega,T}).$ As $\varepsilon\rightarrow0,$ $z^{\varepsilon}$ converges to
$z$ in $L^{s}(Q)$ for any $s\geqq1$, and $z$ is a solution of the heat
equation in $\mathcal{D}^{\prime}(Q^{\prime}),$ hence also in $\mathcal{D}%
^{\prime}(Q_{\Omega,T}))$. Then $z\in C^{\infty}(Q_{\Omega,T}),$ hence
max$_{Q}z<1,$ thus $u\in C^{\infty}(Q_{\Omega,T}).$ And $\left\Vert
z\right\Vert _{L^{\infty}(Q^{\prime})}<1-e^{-\left\Vert u\right\Vert
_{L^{\infty}(Q^{\prime})}}$, then (\ref{gff}) follows from analogous estimates
on $z$.\medskip

(ii) From the estimate (\ref{gff}), one can extract a diagonal subsequence,
converging $a.e.$ to a function $u$ in $Q_{\Omega,T}$, and the convergence
holds in $C_{loc}^{2,1}(Q_{\Omega,T}).$ Then $u$ is a weak solution of
(\ref{un}) in $Q_{\Omega,T}.$\medskip
\end{proof}

In the case of the Dirichlet problem we obtain a corresponding regularity
result for the \textit{bounded} solutions. Our proof can be compared to the
proof of \cite[Proposition 4.1]{BeDa} relative to the case $q<1.$

\begin{theorem}
\label{Dir} Let $1<q\leqq2.$ Let $\Omega$ be a smooth bounded domain. Let $u$
be any weak nonnegative solution of problem $(D_{\Omega,T})$, such that $u\in
L_{loc}^{\infty}(\left(  0,T\right)  ;L^{\infty}(\Omega))$.\medskip

\noindent(i) Then $u$ satisfies the local estimates of Theorem \ref{T.2.1}.
Moreover, $u\in C^{1,0}(\overline{\Omega}\times(0,T))$ and there exists
$\gamma\in\left(  0,1\right)  $ such that, for any $0<s<\tau<T,$
\begin{equation}
\left\Vert u\right\Vert _{C(\overline{\Omega}\times\left[  s,\tau\right]
)}+\left\Vert \nabla u\right\Vert _{C^{\gamma,\gamma/2}(\overline{\Omega
}\times\left[  s,\tau\right]  )}\leqq C\Phi(\left\Vert u\right\Vert
_{L^{\infty}(Q_{\Omega,s/2,\tau})}) \label{gou}%
\end{equation}
where $C=C((N,q,\Omega,s,\tau,\gamma),$ and $\Phi$ is an increasing
function.\medskip

\noindent(ii) For any sequence $(u_{n})$ of weak solutions of $(D_{\Omega,T})$
uniformly bounded in $L_{loc}^{\infty}((0,T);L^{\infty}\left(  \Omega\right)
),$ one can extract a subsequence converging in $C_{loc}^{1,0}(\overline
{\Omega}\times(0,T))$ to a weak solution $u$ of $(D_{\Omega,T})$.
\end{theorem}

\begin{proof}
(i) $\bullet$ Case $q<2.$ From Lemma \ref{subdir}, we have $\nabla u\in
L_{loc}^{2}(0,T);L^{2}\left(  \Omega\right)  )$ and $u\in C((0,T);L^{1}%
(\Omega)).$ Then $f=-|\nabla u|^{q}\in L_{loc}^{q_{1}}((0,t);L^{q_{1}}\left(
\Omega\right)  ).$ For any $0<s<\tau<T,$ and $t\in\left[  s/2,\tau\right]  ,$
we can write $u(.,t)=u_{1}(.,t)+u_{2}(.,t),$ from Lemma \ref{ent}, where
\[
u_{1}(.,t)=e^{(t-s/2)\Delta}u(\frac{s}{2}),\qquad u_{2}(.,t)=\int_{s/2}%
^{t}e^{(t-\sigma)\Delta}f(\sigma)d\sigma.
\]
We get $u_{1}\in C^{\infty}(\overline{Q_{\Omega,s,\tau}})$ from the
regularizing effect of the heat equation, and $u_{2}\in\mathcal{W}^{2,1,q_{1}%
}(Q_{\Omega,T}),$ from \cite[ theorem IV.$9.1$]{LSU}. As above, from the
Gagliardo estimate, we get $f\in L_{loc}^{q_{2}}((0,t);L^{q_{2}}\left(
\Omega\right)  ),$ and by induction $|\nabla u|\in C^{\gamma,\gamma
/2}(\overline{Q_{\Omega,s,\tau}})$ for some $\gamma\in(0,1),$ see \cite[Lemma
II.3.3]{LSU}. The estimates follow as above.\medskip

$\bullet$ Case $q=2$. From Theorem \ref{T.2.1}, $u$ is smooth in $Q_{\Omega
,T}$, and $z=1-e^{-u}$ is a solution of the heat equation, and $z\in
C((0,T);L^{1}(\Omega)).$ Then $z(.,t)=e^{(t-s/2)\Delta}z(s/2),$ thus $z\in
C^{\infty}(\overline{Q_{\Omega,s,\tau}}).$ This implies that max$_{\overline
{Q_{\Omega,s,\tau}}}z<1,$ thus $u\in C^{\infty}(\overline{Q_{\Omega,s,\tau}})$
and the estimates follow again.\medskip

(ii) It follows directly from (\ref{gou}).
\end{proof}

\begin{remark}
\label{chos}As a consequence, in the case $q\leqq2,$ we find again the
estimate (\ref{univ}) for the problem $(D_{\Omega,T})$ without using the
Bernstein argument, and it is valid for any weak solution $u\in L_{loc}%
^{\infty}(\left(  0,T\right)  ;L^{\infty}(\Omega)).$
\end{remark}

\subsection{Singular solutions or supersolutions\bigskip}

In the study some functions play a fundamental role. The first one was
introduced in \cite{BeLa99}.

\subsubsection{A stationary supersolution}

Assume that $1<q<2.$ Equation (\ref{un}) admits a stationary solution whenever
$N=1$ or $N\geqq2,$ $1<q<N/(N-1),$ defined by
\[
s\in\left(  0,\infty\right)  \longmapsto\Gamma_{N}(s)=\gamma_{N,q}%
s^{-a},\qquad\hspace{0.06in}a=\frac{2-q}{q-1},\qquad\gamma_{N,q}%
=a^{-1}(a+2-N)^{1-q}.
\]
Moreover in the range $1<q<2,$ the function $\Gamma=\Gamma_{1}$ defined by
\begin{equation}
s\in\left(  0,\infty\right)  \longmapsto\Gamma(s)=\gamma_{q}s^{-a}%
,\qquad\hspace{0.06in}a=\frac{2-q}{q-1},\qquad\gamma_{q}=\frac{(q-1)^{-a}%
}{2-q}, \label{ama}%
\end{equation}
is a radial supersolution of equation (\ref{un}) for any $N.$

\subsubsection{Large solutions}

Here we recall a main result of \cite{CLS} obtained as a consequence of the
universal estimates.

\begin{theorem}
[\cite{CLS}]\label{L.3.3}Let $G$ be any smooth bounded domain, and $\eta>0$
such that $B_{\eta}\subset\subset G.$ Then for any $q>1,$ there exists a
(unique) solution $Y_{\eta}^{G}$ of the problem%
\begin{equation}
\left\{
\begin{array}
[c]{c}%
(Y_{\eta}^{G})_{t}-\Delta Y_{\eta}^{G}+|\nabla Y_{\eta}^{G}|^{q}%
=0,\qquad\text{in}\hspace{0.05in}Q_{G,\infty},\\
Y_{\eta}^{G}=0,\qquad\text{on }\hspace{0.05in}\partial G\times(0,\infty),\\
Y_{\eta}^{G}(x,0)=\left\{
\begin{array}
[c]{c}%
\infty\quad\text{if }x\in B_{\eta},\\
0\quad\text{if not,}%
\end{array}
\right.
\end{array}
\right.  \label{trg}%
\end{equation}
which is uniformly Lipschitz continuous in $\overline{G}$ for $t$ in compacts
sets of $(0,\infty)$ and is a classical solution of the problem for $t>0,$ and
satisfies the initial condition in the sense:%
\begin{equation}
\lim_{t\rightarrow0}\inf_{x\in K}Y_{\eta}^{G}(x,t)=\infty,\quad\forall K\text{
compact}\subset B_{\eta};\qquad\lim_{t\rightarrow0}\sup_{x\in K}Y_{\eta}%
^{G}(x,t)=0,\quad\forall K\text{ compact}\subset\overline{G}\backslash
\overline{B_{\eta}}. \label{mil}%
\end{equation}
And $Y_{\eta}^{G}$ is the supremum of the solutions $y_{\varphi_{\eta,G}}$
with nonnegative initial data $\varphi_{\eta,G}\in C(G)$ such that
$\varphi_{\eta,G}=0$ on $G\backslash\overline{B_{\eta}}.$
\end{theorem}

A crucial point for existence was the construction of a supersolution for the
problem in a ball:

\begin{lemma}
\label{super}For any ball $B_{s}\subset\mathbb{R}^{N}$ and any $\lambda>0,$
there exists a supersolution $w_{\lambda,s}$ of equation (\ref{un}) in
$B_{s}\times\left[  0,\infty\right)  $, such that%
\[
w_{\lambda,s}=\infty\text{ \ on }\partial B_{s}\times\left[  0,\infty\right)
,\qquad w_{\lambda,s}=\lambda e^{ct+1/\alpha_{s}(x)},\quad c=c(\lambda)>0,
\]
where $\alpha_{s}$ is the solution of $-\Delta\alpha_{s}=1$ in $B_{s}$ and
$\alpha_{s}=0$ on $\partial B_{s}.$
\end{lemma}

\subsection{Some trace results}

First we extend a trace result of \cite{MouV}.

\begin{lemma}
\label{L.4.3}Let $U\in C((0,T);L_{loc}^{1}(\Omega))$ be any nonnegative weak
solution of equation
\begin{equation}
U_{t}-\Delta U=\Phi\label{4.7}%
\end{equation}
in $Q_{\Omega,T}$, with $\Phi\in L_{loc}^{1}(Q_{\Omega,T})$.\medskip

\noindent(i) Assume that $\Phi\geqq-F,$ where $F\in L_{loc}^{1}(\Omega
\times\lbrack0,T))$. Then $U(.,t)$ converges weak$^{\ast}$ to some Radon
measure $U_{0}:$%
\[
\lim_{t\rightarrow0}\int_{\Omega}U(.,t)\varphi dx=\int_{\Omega}\varphi
dU_{0},\qquad\forall\varphi\in C_{c}(\Omega).
\]
Furthermore, $\Phi\in L_{loc}^{1}([0,T);L_{loc}^{1}(\Omega)),$ and for any
$\varphi\in C_{c}^{2}(\Omega\times\lbrack0,T))$,
\begin{equation}
-\int_{0}^{T}\int_{\Omega}(U\varphi_{t}+U\Delta\varphi+\Phi\varphi
)dxdt=\int_{\Omega}\varphi(.,0)dU_{0}. \label{tru}%
\end{equation}
(ii) Assume that $\Phi$ has a constant sign. Then
\begin{equation}
\Phi\in L_{loc}^{1}([0,T);L_{loc}^{1}(\Omega))\Longleftrightarrow U\in
L_{loc}^{\infty}{(}\left[  0,T\right)  {;L_{loc}^{1}(}\Omega)). \label{eki}%
\end{equation}

\end{lemma}

\begin{proof}
(i) Let $\omega\subset\subset\omega^{\prime}\subset\subset\Omega$ and
$0<s<\tau<T$. We approximate $U$ by $U_{\varepsilon}$ and set $\Phi
+F=E\geqq0,$ so that for $\varepsilon$ small enough,
\[
(U_{\varepsilon})_{t}-\Delta U_{\varepsilon}=E_{\varepsilon}-F_{\varepsilon
}\qquad\text{in }Q_{\omega^{\prime},s/2,\tau}.
\]
Let $\phi_{1}$ be a positive eigenfunction associated to the first eigenvalue
$\lambda_{1}$ of $-\Delta$ in $W_{0}^{1,2}(\omega)$. Multiplying the equation
by $\phi_{1}$ and integrating on over $\omega$, we get, for any $t\in
(s/2,\tau)$,
\[
\frac{d}{dt}\int_{\omega}U_{\varepsilon}(.,t)\phi_{1}dx+\lambda_{1}%
\int_{\omega}U_{\varepsilon}(.,t)\phi_{1}dx=-\int_{\partial\omega
}U_{\varepsilon}(.,t)\frac{\partial\phi_{1}}{\partial\nu}d\sigma+\int_{\omega
}E_{\varepsilon}(.,t)\phi_{1}dx-\int_{\omega}F_{\varepsilon}(.,t)\phi_{1}dx.
\]
We set
\begin{align*}
X(t)  &  =\int_{\omega}U(.,t)\phi_{1}dx,\qquad h(t)=e^{\lambda_{1}t}%
X(t)-\int_{t}^{\tau}\int_{\omega}e^{\lambda_{1}s}F(.,s)\phi_{1}dxd\theta,\\
X_{\varepsilon}(t)  &  =\int_{\omega}U_{\varepsilon}(.,t)\phi_{1}dx,\qquad
h_{\varepsilon}(t)=e^{\lambda_{1}t}X_{\varepsilon}(t)-\int_{t}^{\tau}%
\int_{\omega}e^{\lambda_{1}s}F_{\varepsilon}(.,s)\phi_{1}dxd\theta.
\end{align*}
Then $h_{\varepsilon}$ is nondecreasing on $\left(  s/2,\tau\right)  ,$ and
then $h_{\varepsilon}(\tau)\geqq h_{\varepsilon}(s).$ On the other hand,
$X_{\varepsilon}(t)$ converges to $X(t)$ a.e. $in$ $\left(  0,T\right)  $ as
$\varepsilon\rightarrow0$. Since $U\in C((0,T);L_{loc}^{1}(\Omega)),$ we
deduce that $h(\tau)\geqq h(s)+\int_{s}^{\tau}\int_{\omega}E(.,t)\phi_{1}dx$.
Thus $h$ is nondecreasing on $(0,T)$. From the assumption on $F,$ $X$ has a
limit as $t\rightarrow0$, and $\Phi\in L_{loc}^{1}([0,T);L_{loc}^{1}%
(\Omega)).$Otherwise, for any nonnegative $\psi\in\mathcal{C}_{c}^{2}(\Omega
)$, for any $t<\tau,$ there holds
\begin{equation}
\int_{\Omega}U(.,\tau)\psi dx-\int_{t}^{\tau}\int_{\Omega}(U\Delta\psi
+\Phi\psi)dxdt=\int_{\Omega}U(.,t)\psi dx \label{lir}%
\end{equation}
from (\ref{gir}). Thus $\int_{\Omega}U(.,t)\psi dx$ has a nonnegative limit
$\mu(\psi)$ as $t\rightarrow0,$ and
\[
\int_{\Omega}U(.,\tau)\psi dx-\int_{0}^{\tau}\int_{\Omega}(U\Delta\psi
+\Phi\psi)dxdt=\mu(\psi)
\]
Then $\mu$ is a nonnegative linear functional on $\mathcal{C}_{c}^{2}%
(\Omega),$ thus it extends in a unique way as a Radon measure $u_{0}$ on
$\Omega.$ Finally for any $\varphi\in C_{c}^{\infty}(\Omega\times\left[
0,T\right]  )$, we have
\[
-\int_{t}^{T}\int_{\Omega}(U\varphi_{t}+U\Delta\varphi+\Phi\varphi
)dxdt=\int_{\Omega}U(.,t)\varphi(.,t)dx.
\]
Going to the limit as $t\rightarrow0,$ we deduce (\ref{tru}), since
\[
\left\vert \int_{\Omega}U(.,t)(\varphi(.,t)-\varphi(.,0))dx\right\vert \leqq
Ct\int_{\text{supp}\varphi}U(.,t)dx.
\]
(ii) If $U\in L_{loc}^{\infty}{(}\left[  0,T\right)  {;L_{loc}^{1}(}\Omega)),$
then $\int_{t}^{\tau}\int_{\Omega}\Phi\psi dxdt$ is bounded as $t\rightarrow
0,$ and $\Phi\in L_{loc}^{1}([0,T);L_{loc}^{1}(\Omega))$ from the Fatou Lemma.
The converse is a direct consequence of (i). \medskip
\end{proof}

We deduce a trace property for equation (\ref{un}), inspired by the results of
\cite{MaVe} for equation \ref{bf}, see also \cite{BiChVe}:

\begin{proposition}
\label{dic} For any nonnegative weak solution $u$ of (\ref{un}) in
$Q_{\Omega,T}$, the following conditions are equivalent:\medskip

\noindent(i) $u\in L_{loc}^{\infty}{(}\left[  0,T\right)  {;L_{loc}^{1}%
(}\Omega)),$\medskip

\noindent(ii) $\nabla u\in L_{loc}^{q}(\Omega\times\left[  0,T\right)
)$,\medskip

\noindent(iii) $u(.,t)$ converges weak$^{\ast}$ to some nonnegative Radon
measure $u_{0}$ in $\Omega.$

\noindent And then for any $\tau\in(0,T),$ and any $\varphi\in C_{c}%
^{1}(\Omega\times\left[  0,T\right)  )$,
\begin{equation}
\int_{\Omega}u(.,\tau)\varphi dx+\int_{0}^{\tau}\int_{\Omega}(-u\varphi
_{t}+\nabla u.\nabla\varphi-\left\vert \nabla u\right\vert ^{q}\varphi
)dxdt=\int_{\Omega}\varphi(.,0)du_{0}. \label{hou}%
\end{equation}

\end{proposition}

\begin{remark}
If $q\geqq2,$ and $u$ admits a Radon measure $u_{0}$ as a trace, in the sense
of condition (iii), then necessarily
\[
u_{0}\in L_{loc}^{1}(\Omega),\quad\text{and }u\in C\left(  \left[  0,T\right)
;L_{loc}^{1}(\Omega)\right)  .
\]
Indeed condition (ii) implies that $u\in L_{loc}^{2}(\left[  0,T\right)
;W_{loc}^{1,2}(\Omega)),$ and $u_{t}\in L_{loc}^{2}((0,T);W_{loc}%
^{-1,2}(\Omega))+L^{1}\left(  Q_{\omega,T}\right)  ,$ then the conclusion
holds from \cite{Po}. As a first consequence, there exists no weak solution of
equation (\ref{un}) with a a Dirac mass as initial data. This had been shown
in \cite[Theorem 2.2 and Remark 2.1]{Al} for the Dirichlet problem
$(D_{\Omega,T}).$.
\end{remark}

\subsection{Behaviour of Solutions of (\ref{un}), (\ref{R}) in $\Omega_{0}$}

Next we come to problem (\ref{un}), (\ref{R}). In order to see what occurs at
$t=0,$ we extend the solutions on $(-T,T)$ as in \cite{BrFr}.

\begin{proposition}
\label{three}Let $u$ be any weak solution of (\ref{un}), (\ref{R}). Then the
function $\overline{u}$ defined a.e. in $Q_{\Omega,-T,T}$ by
\begin{equation}
\overline{u}(x,t)=\left\{
\begin{array}
[c]{cc}%
u(x,t), & \text{if }\hspace{0.05in}(x,t)\in Q_{\Omega,T},\\
0 & \text{if }\hspace{0.05in}(x,t)\in Q_{\Omega,-T,0},
\end{array}
\right.  \label{uba}%
\end{equation}
is a weak solution of the equation (\ref{un}) in $Q_{\Omega_{0},-T,T}$. If
moreover
\begin{equation}
\lim_{t\rightarrow0}\int_{\Omega}u(.,t)\varphi dx=0,\qquad\forall\varphi\in
C_{c}(\Omega), \label{bra}%
\end{equation}
then $\overline{u}$ is a weak solution of (\ref{un}) in $Q_{\Omega,-T,T}.$
\end{proposition}

\begin{proof}
By assumption, $u\in L_{loc}^{1}(\left[  0,T\right)  \times\Omega_{0})$, hence
$\overline{u}\in{L_{loc}^{1}(}Q_{\Omega_{0},-T,T}).$ Then we can define
$\nabla\overline{u}\in\mathcal{D}^{\prime}(Q_{\Omega_{0},-T,T})$ and for any
$\varphi\in\mathcal{D}(Q_{\Omega_{0},-T,T}),$%
\[
<\nabla\overline{u},\varphi>=-\int_{-T}^{T}\int_{\Omega}\overline{u}%
\nabla\varphi dxdt=-\int_{0}^{T}\int_{\Omega}u\nabla\varphi dxdt.
\]
For any $k\geqq1,$ we consider a function $\zeta_{k}$ on $\left[
0,\infty\right)  $ such that
\begin{equation}
\zeta_{k}(t)=\zeta(kt),\text{ where }\zeta\in C^{\infty}\left(  \left[
0,\infty\right)  \right)  ,\quad\zeta(\left[  0,\infty\right)  )\subset\left[
0,1\right]  ,\quad\zeta\equiv0\text{ in }\left[  0,1\right]  ,\quad\zeta
\equiv1\text{ in }\left[  2,\infty\right)  . \label{eta}%
\end{equation}
Since $u$ is a weak solution of (\ref{un}), there holds%
\begin{equation}
-\int_{0}^{T}\int_{\Omega}u\nabla(\varphi\zeta_{k})dxdt=\int_{0}^{T}%
\int_{\Omega}\varphi\zeta_{k}\nabla udxdt. \label{aut}%
\end{equation}
From (\ref{R}), we see that $u\in L_{loc}^{\infty}{(}\left[  0,T\right)
{;L_{loc}^{1}(}\Omega_{0})),$ hence $\left\vert \nabla u\right\vert \in
L_{loc}^{q}(\Omega_{0}\times\left[  0,T\right)  ),$ from Proposition
\ref{dic}. Then we can go to the limit in (\ref{aut}) as $k\rightarrow\infty$
from the Lebesgue theorem, hence
\[
-\int_{0}^{T}\int_{\Omega}u\nabla\varphi dxdt=\int_{0}^{T}\int_{\Omega}%
\varphi\nabla udxdt.
\]
Thus $\nabla\overline{u}\in{L_{loc}^{q}(}Q_{\Omega_{0},-T,T})$ and
$\nabla\overline{u}(x,t)=\chi_{\left(  0,T\right)  }\nabla u(x,t)$; hence also
$\nabla\overline{u}\in$ ${L_{loc}^{2}(}Q_{\Omega_{0},-T,T})$ from Lemma
\ref{subcal}, and for any $\varphi\in\mathcal{D}(Q_{\Omega_{0},-T,T}),$
\begin{equation}
\int_{-T}^{T}\int_{\Omega}(-\overline{u}\varphi_{t}+\nabla\overline{u}%
.\nabla\varphi+|\nabla\overline{u}|^{q}\varphi)dxdt=\int_{0}^{T}\int_{\Omega
}(-u\varphi_{t}+\nabla u.\nabla\varphi+|\nabla u|^{q}\varphi)dxdt. \label{cho}%
\end{equation}
Moreover
\begin{align*}
0  &  =\int_{0}^{T}\int_{\Omega}(-u(\varphi\zeta_{k})_{t}+\nabla
u.\nabla(\varphi\zeta_{k})+|\nabla u|^{q}\varphi\zeta_{k}dxdt\\
&  =-\int_{0}^{T}\int_{\Omega}u\varphi(\zeta_{k})_{t}dxdt+\int_{0}^{T}%
\int_{\Omega}(-u\varphi_{t}\zeta_{k}+\nabla u.\nabla(\varphi\zeta_{k})+|\nabla
u|^{q}\varphi\zeta_{k}dxdt.
\end{align*}
As $k\rightarrow\infty,$ the first term in the right hand side tends to $0$
from (\ref{R}), since
\begin{equation}
\left\vert \int_{0}^{T}\int_{\Omega}u\varphi(\zeta_{k})_{t}dxdt\right\vert
\leqq Ck\int_{1/k}^{2/k}\int_{\Omega}u\varphi dxdt\leqq C\sup_{t\in\left[
1/k,2/k\right]  }\int_{\text{supp}\varphi}u(.,t)dx, \label{chu}%
\end{equation}
and we can go to the limit in the second term, since $\left\vert \nabla
u\right\vert \in L_{loc}^{q}(\Omega_{0}\times\left[  0,T\right)  )$. Thus from
(\ref{cho}), $\overline{u}$ is a weak solution of equation (\ref{un}) in
$Q_{\Omega_{0},-T,T}$. If (\ref{bra}) holds, the same result holds in $\Omega$
instead of $\Omega_{0}.$\medskip
\end{proof}

\begin{corollary}
\label{sous}Assume $1<q\leqq2.$ Then any weak solution $u$ of (\ref{un}),
(\ref{R}) satisfies $u\in C^{2,1}(\Omega_{0}\times\left[  0,T\right)  )$ and
$u(x,0)=0,\;\forall x\in\Omega_{0}.$ \medskip

If (\ref{bra}) holds, then $u\in C^{2,1}(\Omega\times\left[  0,T\right)  )$
and $u(x,0)=0,\;\forall x\in\Omega.$
\end{corollary}

\begin{proof}
It follows directly from From Proposition \ref{three} and Theorem \ref{T.2.1}
applied to $\overline{u}.$
\end{proof}

\section{The critical or supercritical case\label{3}}

\subsection{Removability in the range $q_{\ast}<q<2$}

For any $1<q<2$ we can compare the solutions with the function $\Gamma$
defined at (\ref{ama}).

\begin{lemma}
\label{L.3.2}Let $1<q<2.$ Let $u$ be any nonnegative weak solution of
(\ref{un}) in $Q_{\Omega,T}$, satisfying (\ref{R}).

\noindent(i) Let $r>0$ such that $B_{r}\subset\Omega.$Then there exists
$\tau_{1}>0$ (depending on $u,r)$ such that
\begin{equation}
0\leqq u(x,t)\leqq\Gamma(\left\vert x\right\vert )\qquad\forall(x,t)\in
Q_{B_{r\backslash\left\{  0\right\}  ,\tau_{1}}}. \label{3.3}%
\end{equation}
(ii) If $\Omega=\mathbb{R}^{N},$ then
\begin{equation}
0\leqq u(x,t)\leqq\Gamma(\left\vert x\right\vert )\qquad\forall(x,t)\in
Q_{_{\mathbb{R}^{N}\backslash\left\{  0\right\}  ,\tau_{1}}}. \label{gam}%
\end{equation}

\end{lemma}

\begin{proof}
(i) For any $\eta\in(0,r)$, we put $\Omega_{\eta}=B_{r}\backslash
\overline{B_{\eta}},$ and we set $F_{\eta}(x)=\Gamma(\left\vert x\right\vert
-\eta),$ for any $x\in\Omega_{\eta}.$ We find
\[
-\Delta F_{\eta}+|\nabla F_{\eta}|^{q}=\gamma_{q}a\frac{(N-1)}{\left\vert
x\right\vert }(\left\vert x\right\vert -\eta)^{-(a+1)}\geqq0,
\]
thus $F_{\eta}$ is a super-solution of (\ref{un}) in $Q_{\Omega_{\eta},\infty
}$. From Theorem \ref{T.2.1} and Proposition \ref{three}, $u\in\mathcal{C}%
^{2,1}(Q_{\Omega,T})\cap C(\Omega_{0}\times\left[  0,T\right)  )$ and
$u(.,0)=0$. Then there exists $\tau_{1}<T$ such that $\max_{\underset
{\left\vert x\right\vert =r}{t\in\left[  0,\tau_{1}\right]  }}u(t,x)<1$, and
$u$ is bounded in $\overline{\Omega_{\eta}}\times\left[  0,\tau_{1}\right]  $.
For any $\varepsilon>0$ small enough, we have $u(x,t)\leqq F_{\eta}(x)$ on
$\partial B_{\eta+\varepsilon}\times\left[  0,\tau_{1}\right]  .$ From the
comparison principle in $Q_{\Omega_{\eta+\varepsilon},\tau_{1}}$, we get
$u(x,t)\leqq F_{\eta}(x)$ in $\Omega_{\eta}\times\left[  0,\tau_{1}\right]  ,$
as $\varepsilon\rightarrow0.$ As $\eta\rightarrow0$, we deduce (\ref{3.3}%
).\medskip

(ii) From Lemma \ref{super}, for any $x_{0}\in\mathbb{R}^{N}\backslash B_{2},$
the function $x\mapsto w_{1,1}(x-x_{0})$ is a supersolution of equation
(\ref{un}) in $Q_{B(x_{0},1),\infty}$, then in particular $u(t,x_{0})\leqq
e^{c(1)t+1/\alpha_{1}(0)},$ thus $u$ bounded in $Q_{\mathbb{R}^{N}\backslash
B_{2},T}.$ From the comparison principle in $\mathbb{R}^{N}\backslash
\overline{B_{\eta}}$ for any $\eta\in(0,1),$ see \cite{GiGuKe}, we find
$u(x,t)\leqq F_{\eta}(x)$ in $Q_{\mathbb{R}^{N}\backslash\overline{B_{\eta}%
},T},$ hence (\ref{gam}) holds as $\eta\rightarrow0.$\medskip
\end{proof}

As a direct consequence we get a simple proof of Theorem \ref{stop} in case
$q_{\ast}<q<2:$

\begin{theorem}
\label{T.3.1} Let $q_{\ast}<q<2$. Suppose that $u$ is a nonnegative weak
solution of (\ref{un}),(\ref{R}). Then $u\in C(\Omega\times\left[  0,T\right)
)$ and $u(x,0)=0,\;\forall x\in\Omega.$
\end{theorem}

\begin{proof}
The assumption $q_{\ast}<q$ is equivalent to $a<N.$ Let $B_{r}\subset\Omega$
and $\tau_{1}$ defined at Lemma \ref{L.3.2}; we find for any $t\in\left(
0,\tau_{1}\right)  ,$
\[
\int_{B_{r}}u(.,t)dx\leqq\int_{B_{r}}\Gamma(\left\vert x\right\vert
)dx\leqq\frac{\gamma_{q}\left\vert \partial B_{1}\right\vert r^{N-a}}{N-a};
\]
then $u\in L^{\infty}((0,\tau_{1});L^{1}(B_{r})).$ Applying Proposition
\ref{dic}, $u(.,t)$ converges weak$^{\ast}$ to a measure $\mu$ on $B_{r}:$%
\[
\lim_{t\rightarrow0}\int_{B_{r}}u(.,t)\psi dx=\int_{B_{r}}\psi d\mu
,\qquad\forall\psi\in C_{c}\left(  B_{r}\right)  .
\]
From (\ref{R}), $\mu$ is concentrated at $0$ and then $\mu=k\delta_{0}$ for
some $k\geqq0.$ Suppose that $k>0,$ we choose $\psi_{\eta}$ such that
$\psi_{\eta}(0)=1,$ $\psi_{\eta}(B_{r})\subset\left[  0,1\right]  ,$
supp$\psi_{\eta}\subset B_{\eta},$ with $\eta\in\left(  0,r\right)  $ small
enough such that $\gamma_{q}\left\vert \partial B_{1}\right\vert \eta
^{N-a}\leqq(N-a)k/2$. For any $t\in(0,\tau_{1})$, lemma \ref{L.3.2} yields
\begin{equation}
\int_{B_{r}}u(.,t)\psi_{\eta}dx\leqq\int_{B\eta}\Gamma(\left\vert x\right\vert
)dx\leqq\frac{k}{2}. \label{3.10}%
\end{equation}
As $t$ tends to $0$ the left-hand side tends to $k,$ which is a contradiction.
Then $k=0,$ hence for any $\psi\in C_{c}^{\infty}\left(  B_{r}\right)  ,$%
\begin{equation}
\lim_{t\rightarrow0}\int_{B_{r}}u(.,t)\psi dx=0, \label{bou}%
\end{equation}
and we conclude from Corollary \ref{sous}.
\end{proof}

\subsection{Removability in the whole range $q_{\ast}\leqq q<2$}

The proof of Theorem \ref{T.3.1} is not valid in the critical case $q=q_{\ast
},$ since the function $x\longmapsto\Gamma(\left\vert x\right\vert
)=\gamma_{q}\left\vert x\right\vert ^{-N}$ is not integrable near $0.$ Then we
use another argument of comparison with the large solutions constructed at
Theorem \ref{L.3.3}, valid for any $1<q<2:$

\begin{proposition}
\label{L.3.4} Let $1<q<2$. Under the assumptions of Theorem \ref{L.3.3} with
$G=B_{n}$ $(n\geqq1)$ the functions $Y_{\eta}^{B_{n}}$ converge as
$n\rightarrow\infty$ to a radial solution $Y_{\eta}$ of problem
\begin{equation}
\left\{
\begin{array}
[c]{c}%
(Y_{\eta})_{t}-\Delta Y_{\eta}+|\nabla Y_{\eta}|^{q}=0,\qquad\text{in}%
\hspace{0.05in}Q_{\infty},\\
Y_{\eta}(x,0)=\left\{
\begin{array}
[c]{c}%
\infty\quad\text{if }x\in B_{\eta},\\
0\quad\text{if not.}%
\end{array}
\right.
\end{array}
\right.  \label{py}%
\end{equation}
Then, as $\eta\rightarrow0,$ $Y_{\eta}$ converges to a radial self-similar
solution $Y$ of equation (\ref{un}) in $Q_{\mathbb{R}^{N},\infty}$, such that%
\begin{equation}
Y(x,t)\leqq\Gamma\left(  \left\vert x\right\vert \right)  ,\qquad\text{ in
}Q_{\infty}, \label{enf}%
\end{equation}%
\begin{equation}
Y(x,t)\leqq C(1+t^{-\frac{1}{q-1}}),\qquad\text{ in }Q_{\infty}, \label{vir}%
\end{equation}
where $C=C(N,q),$ and
\begin{equation}
\lim_{t\rightarrow0}(\sup_{\left\vert x\right\vert \geqq r}Y(x,t))=0.
\label{dir}%
\end{equation}
If $q_{\ast}\leqq q<2,$ then $Y=0.$
\end{proposition}

\begin{proof}
Let $\eta\in\left(  0,1/2\right)  $. For any $n\geqq1,$ $Y_{\eta}^{B_{n}}$ is
the supremum of the solutions $y_{\varphi_{\eta,B_{n}}}$; from the comparison
principle, since $q<2,$
\begin{equation}
y_{\varphi_{\eta,B_{n}}}(x,t)\leqq\Gamma\left(  \left\vert x\right\vert
-\eta\right)  \qquad\text{in }(B_{n}\backslash\overline{B_{\eta}}%
)\times\left[  0,\infty\right)  . \label{rose}%
\end{equation}
From Lemma \ref{mono} in $Q_{B_{1},\infty},$ we obtain, for any $(x,t)\in
\overline{B_{1}}\times(0,\infty)$
\begin{equation}
y_{\varphi_{\eta,B_{n}}}(x,t)\leqq C(1+t^{-\frac{1}{q-1}})+\gamma_{q}%
\{1-\eta)\}^{-\frac{2-q}{q-1}}\leqq C(1+t^{-\frac{1}{q-1}})+\gamma_{q}%
2^{\frac{2-q}{q-1}}, \label{koi}%
\end{equation}
with $C=C(N,q).$ And for any $(x,t)\in(B_{n}\backslash\overline{B_{1}}%
)\times(0,\infty)$, we have
\begin{equation}
y_{\varphi_{\eta,B_{n}}}(x,t)\leqq\Gamma\left(  \left\vert x\right\vert
-\eta\right)  \leqq\Gamma\left(  1-\eta\right)  \leqq\gamma_{q}2^{\frac
{2-q}{q-1}} \label{flou}%
\end{equation}
Then (\ref{koi}) holds in $B_{n}\times\left[  0,\infty\right)  .$ The same
majoration holds for $Y_{\eta}^{B_{n}}:$ with a new $C=C(N,q),$
\[
Y_{\eta}^{B_{n}}(.,t)\leqq C(1+t^{-\frac{1}{q-1}}),\qquad\text{in }%
Q_{B_{n},\infty}.
\]
Then we can go to the limit as $n\rightarrow\infty,$ for fixed $\eta$. From
Theorem \ref{T.2.1} we can extract a (diagonal) subsequence converging in
$C_{loc}^{2,1}(Q_{\mathbb{R}^{N},\infty})$ to a weak solution $Y_{\eta}$ of
equation (\ref{un})$.$ In fact the whole sequence converges, since $Y_{\eta
}^{B_{n}}\leqq Y_{\eta}^{B_{n+1}}$ in $Q_{B_{n},\infty}.$ Then $Y_{\eta}=\sup
Y_{\eta}^{B_{n}}$ satisfies
\begin{equation}
Y_{\eta}\leqq C(1+t^{-\frac{1}{q-1}}),\qquad\text{ in }Q_{\infty}, \label{vri}%
\end{equation}
and $Y_{\eta}$ solves the problem (\ref{py}) in the sense
\begin{equation}
\lim_{t\rightarrow0}\inf_{x\in K}Y_{\eta}(x,t)=\infty,\quad\forall K\text{
compact}\subset B_{\eta};\qquad\lim_{t\rightarrow0}\sup_{x\in K}Y_{\eta
}(x,t)=0,\quad\forall K\text{ compact}\subset\mathbb{R}^{N}\backslash
\overline{B_{\eta}}. \label{fol}%
\end{equation}
Indeed from Lemma \ref{super}, for any ball $B(x_{0},s)\subset\mathbb{R}%
^{N}\backslash\overline{B_{\eta}},$ and any $\lambda>0,$ we have $Y_{\eta
}^{B_{n}}\leqq w_{\lambda,s}(x-x_{0})$ in $Q_{B(x_{0},s),\infty}$ for any
$n>\left\vert x_{0}\right\vert +\left\vert r\right\vert ;$ in turn $Y_{\eta
}\leqq w_{\lambda,s}(x-x_{0}),$ hence lim$_{t\rightarrow0}\sup_{B(x_{0}%
,s/2)}Y_{\eta}(.,t)\leqq\lambda e^{1/\alpha(s/2)}$ for any $\lambda>0.$
Moreover (\ref{rose}) implies that
\begin{equation}
Y_{\eta}(x,t)\leqq\Gamma\left(  \left\vert x\right\vert -\eta\right)
\qquad\text{in }Q_{\mathbb{R}^{N}\backslash\overline{B_{\eta}},\infty}.
\label{enfi}%
\end{equation}
Then for any $r>\eta,$ and any $p>r,$
\[
\sup_{\left\vert x\right\vert \geqq r}Y_{\eta}(x,t)\leqq\sup_{x\in
B_{p\backslash}\overline{B_{\eta}}}Y_{\eta}(x,t)+\sup_{x\in\mathbb{R}%
^{N}\overline{\backslash B_{p}}}Y_{\eta}(x,t)\leqq\sup_{x\in B_{p\backslash
}\overline{B_{\eta}}}Y_{\eta}(x,t)+\Gamma\left(  \left\vert p\right\vert
-\eta\right)
\]
then we find
\begin{equation}
\lim_{t\rightarrow0}(\sup_{\left\vert x\right\vert \geqq r}Y_{\eta}(x,t))=0,
\label{fil}%
\end{equation}
since $\lim_{r\rightarrow\infty}\Gamma(r)=0.\medskip$

Next we let $\eta\rightarrow0:$ observing that $Y_{\eta}\leqq Y_{\eta^{\prime
}}$ for $\eta\leqq\eta^{\prime},$ in the same way from Theorem \ref{T.2.1},
the function $Y=\inf_{\eta>0}Y_{\eta}$ is a weak solution of equation
(\ref{un}) in $Q_{\mathbb{R}^{N},\infty},$ satisfying the estimates
(\ref{enf}), (\ref{vir}), and (\ref{dir}) which implies in particular
(\ref{VS}). Because of their uniqueness, all the functions $Y_{\eta}^{B_{n}}$
are radial, and satisfy the relation of similarity,%
\[
\kappa^{a}Y_{\eta}^{B_{n}}(\kappa x,\kappa^{2}t)=Y_{\eta/\kappa}^{B_{n/\kappa
}}(x,t),\qquad\forall\kappa>0,\quad\forall(x,t)\in B_{n/k};
\]
then $Y$ is radial and self-similar. \medskip

Suppose $q\geqq q_{\ast}$ and $Y\not \equiv 0;$ writing $Y$ under the similar
form $Y(x,t)=t^{-a/2}f(t^{-1/2}\left\vert x\right\vert ),$ then from
\cite[Theorem 2.1]{QW}, we find \underline{lim}$_{r\rightarrow\infty}%
r^{a}f(r)>0,$ which contradicts (\ref{dir}); thus $Y\equiv0.$\medskip
\end{proof}

\begin{proposition}
\label{muji}Let $1<q<2$. Let $\Omega$ be any domain in $\mathbb{R}^{N}.$ Let
$u$ be any weak solution of (\ref{un}),(\ref{R}) in $Q_{\Omega,T}.$ Then for
any $\tau\in\left(  0,T\right)  $ and any ball $B_{r}\subset\subset\Omega,$
there holds%
\[
u\leqq Y+\max_{\partial B_{r}\times\lbrack0,\tau]}u,\qquad\text{in }%
Q_{B_{r},\tau}.
\]
Moreover, if $\Omega=\mathbb{R}^{N},$ then
\begin{equation}
u\leqq Y,\qquad\text{in }Q_{\mathbb{R}^{N},T} \label{poul}%
\end{equation}
and $u\in C^{2,1}(Q_{\mathbb{R}^{N},\infty})\cap C((0,\infty);C_{b}%
^{2}(\mathbb{R}^{N})).$
\end{proposition}

\begin{proof}
Let $u$ be such a solution in $Q_{\Omega,T}.$ Let $\tau\in\left(  0,T\right)
,$ $B_{r}\subset\subset\Omega$\noindent, and $M_{r}=\max_{\partial B_{r}%
\times\lbrack0,\tau]}u$ and $\varepsilon>0$ be fixed. From Corollary
\ref{sous}, $u\in C(\Omega_{0}\times\left[  0,T\right)  )$ and
$u(x,0)=0,\forall x\in\Omega_{0}.$ Then for any $0<\eta<r/2$, there is
$\delta_{\eta}>0$ such that
\begin{equation}
u(x,t)<\varepsilon,\qquad\text{for }\hspace{0.05in}\eta\leqq\left\vert
x\right\vert \leqq r,\quad t\in(0,\delta_{\eta}). \label{1.26}%
\end{equation}
Let $R>r.$ Next, for any $\delta\in(0,\delta_{\eta})$, we make a comparison in
$Q_{B_{r},\delta,\tau}$ between $u(x,t)$ and
\[
y_{2\eta,R,\delta}(x,t)=Y_{2\eta}^{B_{R}}(x,t-\delta)+M_{r}+\varepsilon
\]
as follows. On the parabolic boundary of $Q_{B_{r},\delta,\tau},$ it is clear
that $u\leqq y_{2\eta,\delta,R},$ since $u\leqq M_{r}$ on $\partial
B_{r}\times\left[  \delta,\tau\right]  ,$ $u(x,\delta)\leqq\varepsilon$ for
$x\in\overline{B_{r}}\backslash\overline{B_{\eta}},$ and $u(x,\delta
)\leqq\infty=y_{2\eta,\delta,R},$ for $x\in\overline{B_{\eta}}$. And
$y_{2\eta,R,\delta}$ converges to $+\infty$ uniformly on $\overline{B_{\eta}}$
as $t\rightarrow\delta,$ and $u(.,\delta)$ is bounded on $\overline{B_{\eta}}%
$. Then, from the comparison principle,
\begin{equation}
u\leqq y_{2\eta,R,\delta},\qquad\text{in }Q_{B_{r},\delta,\tau}. \label{1.27}%
\end{equation}
As $\delta$ tends to $0$ in (\ref{1.27}), and we get
\begin{equation}
u\leqq Y_{2\eta}^{B_{R}}+M_{r}+\varepsilon,\qquad\text{in }Q_{B_{r},\tau},
\label{1.28}%
\end{equation}
by the continuity of $Y_{2\eta}^{B_{R}}$ in $Q_{B_{r},T}$. Since (\ref{1.28})
holds for any $\eta<r/2$, and any $\varepsilon>0,$ we finally obtain
\[
u\leqq Y+M_{r},\qquad\text{in }Q_{B_{r},\tau}.
\]
Moreover if $\Omega=\mathbb{R}^{N},$ then $M_{r}\leqq\Gamma(r)$ from Lemma
\ref{L.3.2}, and we get (\ref{poul}) by letting $r\rightarrow\infty.$ Moreover
$u\in C^{2,1}(Q_{\mathbb{R}^{N},\infty})$ from Theorem \ref{T.2.1}, then from
(\ref{vir}), $u\in C_{b}(Q_{\mathbb{R}^{N},\epsilon,\infty})$ for any
$\epsilon>0,$ then from \cite[Theorems 3 and 6]{GiGuKe}, $u\in C((0,\infty
);C_{b}^{2}(\mathbb{R}^{N})).\medskip$
\end{proof}

As a direct consequence, we deduce a new proof of Theorem \ref{stop}, valid in
the range $q_{\ast}\leqq q<2:$

\begin{theorem}
\label{T.3.4}Let $q_{\ast}\leqq q<2$. Suppose that $u$ is a nonnegative weak
solution of (\ref{un}),(\ref{R}) in $Q_{\Omega,T}.$ \medskip

Then $u\in C(\Omega\times\left[  0,T\right)  )$ and $u(x,0)=0,\;\forall
x\in\Omega.$
\end{theorem}

\begin{proof}
Since $q\geqq q_{\ast},$ we have $Y=0,$ from Proposition \ref{L.3.4}, thus $u$
is bounded in $Q_{B_{r},\tau}$ from Proposition \ref{muji}. Then (\ref{bou})
still holds for any $\psi\in C_{c}^{\infty}\left(  B_{r}\right)  ,$ and we
conclude again from Corollary \ref{sous}.
\end{proof}

\subsection{Removability for $q\geqq2$}

When $q>2,$ the regularity of the solutions of equation (\ref{un}), in
particular the continuity property, is not known up to now. It was shown
recently in \cite{CanCar} that \textit{if }a solution in the viscosity sense
is continuous, then it is H\"{o}lderian. Then it is difficult to apply
comparison theorems. Here we use the transformation $u\longmapsto z=1-e^{-u}%
$\noindent, which reduces classically equation (\ref{un}) to the heat equation
when $q=2$, where we gain the fact that $z$ is bounded. For $p>2,$ our proof
requires regularization arguments.

\begin{theorem}
\label{the3.10}Let $q\geqq2.$ Let $u$ be any weak solution $u$ of equation
(\ref{un}), (\ref{R}), in $Q_{\Omega,T}$.\medskip

\noindent(i) If $q=2,$ then $u\in C^{\infty}(\Omega\times\left[  0,T\right)
),$ and $u(x,0)=0,\quad\forall x\in\Omega.$\medskip

\noindent(ii) If $q>2,$ then $u$ satisfies
\[
\lim_{t\rightarrow0}\int_{\Omega}u(.,t)\varphi dx=0,\qquad\forall\varphi\in
C_{c}(\Omega),
\]
and $u\in C([0,T);L_{loc}^{r}(\Omega))$ for any $r\geqq1$ and $u(.,0)=0$ in
the sense of $L_{loc}^{r}(\Omega)$. Moreover $u\in L^{\infty}(Q_{\omega,\tau
})$ for any $\omega\subset\subset\Omega,$ and $\tau\in\left(  0,T\right)  ,$
and%
\[
\lim_{t\rightarrow0}\sup_{Q_{\omega,t}}u=0.
\]

\end{theorem}

\begin{proof}
Let us set
\begin{equation}
z=1-v,\qquad v=e^{-u}, \label{zv}%
\end{equation}
Notice that $z$ is an increasing function of $u$ and $z$ takes its values in
$\left[  0,1\right]  .\medskip$

(i) Case $q=2.$ From Theorem \ref{T.2.1}, $u$ is a classical solution in
$Q_{\Omega,T}.$ Then $z$ is a classical solution of the heat equation
\[
z_{t}-\Delta z=0
\]
in $Q_{\Omega,T},$ and $z\in C(\Omega_{0}\times\left[  0,T\right)  )$ and
$z(x,0)=0$ for $x\neq0.$ From Lemma \ref{L.4.3}, $z$ converges weak$^{\ast}$
to a Radon measure $\mu$ as $t\rightarrow0,$ necessarily concentrated at $0,$
from (\ref{R}), since $z\leqq u.$ Then $\mu=0,$ because $z$ is bounded. As for
$u,$ defining the extension $\overline{z}$ of $z$ by $0$ for $t\in\left(
-T,0\right)  ,$ we find that $\overline{z}$ is a solution of heat equation in
$Q_{\Omega,-T,T},$ then $\overline{z}\in C^{\infty}(Q_{\Omega,-T,T}).$ Hence
$\overline{z}$ is strictly locally bounded by 1, thus also $\overline{u}\in
C^{\infty}(Q_{\Omega,-T,T}),$ thus $u(0,0)=0,$ and the proof is done.\medskip

(ii) Case $q>2.$ We regularize equation (\ref{un}) and obtain
\[
(u_{\varepsilon})_{t}-\Delta u_{\varepsilon}+(|\nabla u|^{q})_{\varepsilon
}=0,
\]
and we set $v^{\varepsilon}=e^{u_{\varepsilon}}.$Then $v^{\varepsilon}$
satisfies the equation
\[
v_{t}^{\varepsilon}-\Delta v^{\varepsilon}=v^{\varepsilon}\left(  |\nabla
u|^{q})_{\varepsilon}-|\nabla u_{\varepsilon}|^{2}\right)  .
\]
Observe that $v^{\varepsilon}$ is not the regularisation of $v,$ but it has
the same convergence properties. Going to the limit as $\varepsilon
\rightarrow0$, we obtain%
\[
v_{t}-\Delta v=v(|\nabla u|^{q}-|\nabla u|^{2})
\]
in $\mathcal{D}^{\prime}(Q_{\Omega,T}).$ Next we apply lemma \ref{L.4.3} to
$v,$ with
\[
\Phi=v[|\nabla u|^{q}-|\nabla u|^{2}]\in L_{loc}^{1}(Q_{\Omega,T}),\qquad
F=-1,
\]
since from the Young inequality, $\Phi\geqq-v\geqq-1$. Then $z(.,t)$ converges
weak$^{\ast}$ to a Radon measure $\mu$ as $t\rightarrow0,$ and $\Phi\in
L_{loc}^{1}(\Omega\times\lbrack0,T));$ and for any $\varphi\in C_{c}%
^{2}(\Omega\times\lbrack0,T))$ there holds%
\begin{equation}
\int_{0}^{T}\int_{\Omega}z(\varphi_{t}+\Delta\varphi)dxdt=\int_{0}^{T}%
\int_{\Omega}\Phi\varphi dxdt+\int_{\Omega}\varphi(x,0)d\mu,\hspace{0.05in}
\label{moc}%
\end{equation}
from (\ref{tru}). We claim that $\mu=0$ and the extension of $z$ by $0$ for
$t=0$ satisfies
\[
z\in C(\left[  0,T\right)  ,L_{loc}^{1}\left(  \Omega\right)  ).
\]
Indeed, from assumption (\ref{R}), $u(.,t)$converges to $0$ in $L_{loc}%
^{1}\left(  \Omega_{0}\right)  $ as $t\rightarrow0$, thus also $z(.,t).$ For
any sequence $(t_{n})$ tending to $0,$ we can extract a (diagonal) subsequence
such that $u(.,t_{\nu})$ converges to $0,$ a.e. \ in $\Omega.$ Since $z$ is
bounded, it follows that $(z(.,t_{\nu}))$ converges to $0$ in $L_{loc}%
^{1}\left(  \Omega\right)  $ from the Lebesgue theorem. And then $z(.,t)$
converges to $0$ in $L_{loc}^{1}\left(  \Omega\right)  $ as $t\rightarrow
0$.\medskip

\noindent We still consider the extension $\overline{z}$ of $z$ by $0$ on for
$t\in\left(  -T,0\right)  .$ For any $\phi\in\mathcal{D}^{+}(Q_{\Omega
,-T,T}),$ we have from (\ref{moc}),%
\begin{align*}
-\int_{-T}^{T}\int_{\Omega}\overline{z}(\phi_{t}+\Delta\phi)dxdt  &
=-\int_{0}^{T}\int_{\Omega}z(\phi_{t}+\Delta\phi)dxdt=-\int_{0}^{T}%
\int_{\Omega}\Phi\varphi dxdt\\
&  \leqq\int_{0}^{T}\int_{\Omega}(1-z)\varphi dxdt\leqq\int_{-T}^{T}%
\int_{\Omega}(1-\overline{z})\varphi dxdt.
\end{align*}
\medskip Then $\overline{z}$ is a subsolution of equation
\begin{equation}
w_{t}-\Delta w+w=1 \label{flic}%
\end{equation}
in $\mathcal{D}^{\prime}(Q_{\Omega,-T,T}).$ Otherwise $\overline{u}$ is the
weak solution of equation (\ref{un}) in $Q_{\Omega_{0},-T,T}$, then
$\overline{u}$ is subcaloric. As a consequence, for any $\tau\in(0,T)$, and
any ball $B_{2r}\subset\subset\Omega$, the function $\overline{u}$ is
essentially bounded on $Q_{B_{2r}\backslash\overline{B_{r/2}},-\tau,\tau}$ by
a constant $M_{r,\tau},$ and then $\overline{z}\leqq1-e^{-M_{r,\tau}%
}=m_{r,\tau}<1$ on this set. For any $K>0$ the function $y_{K}(t)=1-Ke^{-t}$
is a solution of equation (\ref{flic}). Taking $K=e^{-(M_{r,\tau}+\tau+1)},$
we can apply the comparison principle in $Q_{B_{r},-\tau,\tau}$ to the
regularisation $\overline{z}_{\varepsilon}$ of $\overline{z}$ for
$\varepsilon$ small enough, and deduce that $\overline{z}\leqq y_{K}$ a.e. in
$Q_{B_{r},-\tau,\tau},$ and then
\[
\overline{z}\leqq1-e^{-(M_{r,\tau}+2\tau+1)}<1\text{ \qquad in\ }%
Q_{B_{r},-\tau,\tau}.
\]
Hence $\overline{u}=-\ln(1-\overline{z})$ is essentially bounded in
$Q_{B_{r},-\tau,\tau}.$ Finally $\overline{u}\in L_{loc}^{\infty}%
(Q_{\Omega,-T,T}),$ from the subcaloricity, hence $u\in L_{loc}^{\infty
}(Q_{\Omega,T})$.\medskip

\noindent Besides, for any $0<s<t<\tau,$ and any domain $\omega\subset
\subset\Omega,$%
\[
|u(.,t)-u(.,s)|\leqq e^{\left\Vert \overline{u}\right\Vert _{L^{\infty
}(Q_{\omega,-\tau,\tau})}}|z(.,t)-z(.,s)|;
\]
then $u\in\mathcal{C}([0,T);L_{loc}^{1}(\Omega)),$ and $u\in C([0,T);L_{loc}%
^{r}(\Omega))$, for any $r>1,$ since $u$ is locally bounded.\medskip

\noindent Furthermore, for any ball $B(x_{0},2\rho)\subset\Omega$, and any
$t\in\left(  \rho^{2}-T,T\right)  $,
\[
\sup_{B(x_{0},\rho)\times\left(  t-\rho^{2},t)\right)  }\overline{u}\leqq
C\rho^{-(N+2)}\int_{t-\rho^{2}}^{t}\int_{B(x_{0},2\rho)}\overline{u}dxds,
\]
where $C=C(N),$ see for example \cite[Theorem 6.17]{Li}. Hence for any
$t\in\left(  0,\tau\right)  $ and $\rho<T^{1/2},$ we find
\[
\sup_{B(x_{0},\rho)\times\left(  0,t)\right)  }u\leqq C\rho^{-(N+2)}\int
_{0}^{t}\int_{B(x_{0},2\rho)}udxds\leqq C\rho^{-(N+2)}t\left\Vert u\right\Vert
_{L^{\infty}(Q_{B(x_{0},2\rho),\tau})},
\]
which achieves the proof.
\end{proof}

\subsection{Global removability in $\mathbb{R}^{N}$}

Next we show Theorem \ref{comp} relative to $\Omega=\mathbb{R}^{N}.$ It is a
consequence of Proposition \ref{muji} in case $1<q<2$. In fact the result is
general, as shown below:

\begin{proposition}
\label{pro}Let $q>1.$ Let $u$ be any non-negative weak subsolution of equation
(\ref{un}) in$\hspace{0.05in}Q_{\mathbb{R}^{N},T}$ such that $u\in
C((0,T,L_{loc}^{1}(\mathbb{R}^{N})),$ and
\begin{equation}
\lim_{t\rightarrow0}\int_{\mathbb{R}^{N}}u(.,t)\psi dx=0, \label{col}%
\end{equation}
for any $\psi\in$ $C_{c}\left(  \mathbb{R}^{N}\right)  .$ Then $u\equiv0.$
\end{proposition}

\begin{proof}
From Lemmas \ref{appro} and \ref{subcal}, since $u\in C((0,T,L_{loc}%
^{1}(\mathbb{R}^{N})),$ there holds
\[
\int_{\mathbb{R}^{N}}u(.,t)\psi dx-\int_{\mathbb{R}^{N}}u(.,s)\psi dx+\int
_{s}^{\tau}\int_{\mathbb{R}^{N}}(\nabla u.\nabla\psi+|\nabla u|^{q}\psi
dxdt\leqq0,
\]
for any $\psi\in$ $C_{c}^{2,+}(\mathbb{R}^{N}),$ and any $(s,t)\subset\left(
0,T\right)  .$Taking $\psi=\xi^{q^{\prime}}$ with $\xi\in\mathcal{D}%
^{+}(\mathbb{R}^{N})$ and using H\"{o}lder inequality, we deduce
\begin{align*}
\int_{\mathbb{R}^{N}}u(.,t)\psi dx-\int_{\mathbb{R}^{N}}u(.,s)\psi dx+\int
_{s}^{t}\int_{\mathbb{R}^{N}}|\nabla u|^{q}\psi dxdt  &  \leqq q^{\prime}%
(\int_{s}^{t}\int_{\mathbb{R}^{N}}|\nabla u|^{q}\psi dx)^{\frac{1}{q}}%
(\int_{s}^{t}\int_{\mathbb{R}^{N}}|\nabla\xi|^{q^{\prime}}dx)^{\frac
{1}{q^{\prime}}}\\
&  \leqq\frac{1}{2}\int_{s}^{t}\int_{\mathbb{R}^{N}}|\nabla u|^{q}\psi
dx+C_{q}\int_{s}^{t}\int_{\mathbb{R}^{N}}|\nabla\xi|^{q^{\prime}}dx
\end{align*}
with $C_{q}=(2(q-1))^{q^{\prime}}.$We choose for any $R>r>0$,
\[
\xi(x)=\phi(\frac{\left\vert x\right\vert }{R}),\text{ where }\phi(\left[
0,\infty\right)  )\subset\left[  0,1\right]  ,\quad\phi\equiv1\text{ in
}\left[  0,1\right]  ,\quad\phi\equiv0\text{ in }\left[  2,\infty\right)  ,
\]
and go to the limit as $s\rightarrow0$ from (\ref{col}). It follows that
\begin{equation}
\int_{B_{r}}u(.,t)dx+\frac{1}{2}\int_{0}^{t}\int_{B_{r}}|\nabla u|^{q}%
dxdt\leqq C_{q}tR^{N-q^{\prime}}. \label{II1.35}%
\end{equation}
$\bullet$ First assume $q<N/(N-1);$ then $N-q^{\prime}<0.$ Letting
$R\rightarrow\infty,$ we deduce that $\int_{B_{r}}u(.,t)dx=0,$ for any $r>0$,
thus $u\equiv0.$\medskip

\noindent$\bullet$ Next assume $q\geqq N/(N-1)$. Then we fix some $k\in\left(
1,N/(N-1)\right)  ;$ for any $\eta\in\left(  0,1\right)  $, there holds
$\eta|\nabla u|^{k}\leqq\eta+|\nabla u|^{q},$ hence the function $w_{\eta
}=\eta^{1/(k-1)}(u-\eta t)$ satisfies
\[
(w_{\eta})_{t}-\Delta w_{\eta}+|\nabla w_{\eta}|^{k}\leqq0
\]
in the weak sense. Thanks to Kato's inequality, see for example \cite{Os} or
\cite{BaPi}, we deduce that
\begin{equation}
(w_{\eta}^{+})_{t}-\Delta w_{\eta}^{+}+|\nabla w_{\eta}^{+}|^{k}\leqq0,
\end{equation}
in $\hspace{0.05in}\mathcal{D}^{\prime}(\mathbb{Q}_{\mathbb{R}^{N},T}).$
Moreover $w_{\eta}\in C(\left[  0,T\right)  ,L_{loc}^{1}(\mathbb{R}^{N})),$
and, for any $r>0,$
\[
\lim_{t\rightarrow0^{+}}\int_{B_{r}}w_{\eta}^{+}(.,t)dx=\eta^{-\frac{1}{k-1}%
}\lim_{t\rightarrow0^{+}}\int_{B_{r}}(u(.,t)-\eta t)^{+}dx=0.
\]
By the above proof, $w_{\eta}^{+}\equiv0.$ Letting $\eta$ tend to $0$ we get
again $u\equiv0.$
\end{proof}

\subsection{Behaviour of the approximating sequences}

When $q$ is critical or supercritical, a simple question is to know what can
happen to a sequence of solutions with smooth initial data converging to the
Dirac mass, and one can expect that that it converges to $0.$ We get more
generally the following:

\begin{theorem}
Assume that $q\geqq q_{\ast}.$ Let $\left(  \varphi_{\varepsilon}\right)  $ be
any sequence in $\mathcal{D}^{+}\left(  \mathbb{R}^{N}\right)  $, with
\textrm{supp }$\varphi_{\varepsilon}\in B_{\varepsilon}.$ Then the sequence
$\left(  u_{\varepsilon}\right)  $ of solutions of (\ref{un}) in
$Q_{\mathbb{R}^{N},\infty},$ with inital data $\varphi_{\varepsilon},$
converges to $0$ in $C_{loc}(Q_{\mathbb{R}^{N},\infty}).$ In the same way, if
$\Omega$ is bounded, the sequence $\left(  u_{\varepsilon}^{\Omega}\right)  $
of solutions of $(D_{\Omega,\infty},$ with initial data $\varphi_{\varepsilon
},$ converges to $0$ in $C_{loc}(\overline{\Omega}\times(0,\infty)).$
\end{theorem}

\begin{proof}
Let $\varepsilon\in\left(  0,1\right)  .$ Since $u_{\varepsilon}^{\Omega}\leqq
u_{\varepsilon},$ we only need to prove the result in case $\Omega
=\mathbb{R}^{N}.\medskip$

(i) Case $q<2.$ We use the function $Y_{2\varepsilon}$ defined at (\ref{py}).
There holds $u_{\varepsilon}\leqq Y_{2\varepsilon}$ from the comparison
principle; and $Y_{2\varepsilon}$ converges to $0$ in $C_{loc}^{1}%
(Q_{\mathbb{R}^{N},\infty})$ from Proposition \ref{L.3.4}, then also
$u_{\varepsilon}$.$\medskip$

(ii) Case $q\geqq2.$ Let us fix some $k$ such that $q_{\ast}<k<2.$ As in the
proof of Proposition \ref{pro}, for any $\eta\in\left(  0,1\right)  ,$
$w_{\varepsilon,\eta}=\eta^{1/(k-1)}(u_{\varepsilon}-\eta t)$ satisfies
\begin{equation}
(w_{\varepsilon,\eta})_{t}-\Delta w_{\varepsilon,\eta}+|\nabla w_{\varepsilon
,\eta}|^{k}\leqq0
\end{equation}
in $\hspace{0.05in}\mathcal{D}^{\prime}(Q_{\mathbb{R}^{N},\infty}),$ and
$w_{\varepsilon,\eta}$ $\in L_{loc}^{\infty}(\left[  0,\infty\right)
;L^{\infty}(\mathbb{R}^{N})).$ From the comparison principle we find that
$w_{\varepsilon,\eta}\leqq v_{\varepsilon},$ where $v_{\varepsilon}$ is the
solution of equation (\ref{un}) with $q$ replaced by $k$ and $v_{\varepsilon
}(.,0)=\rho_{\varepsilon};$ hence $u_{\varepsilon}\leqq\eta t+\eta^{1/(k-1)}$.
And $(v_{\varepsilon})$ converges to $0$ in $C_{loc}(Q_{\mathbb{R}^{N},\infty
})$ from (i).\ Let $\mathcal{K=}\left[  s,\tau\right]  \times K$ be any
compact in $Q_{\mathbb{R}^{N},\infty}.$ Then
\[
\lim\sup\left\Vert u_{\varepsilon}\right\Vert _{L^{\infty}(\mathcal{K})}%
\leqq\eta\tau+\eta^{1/(k-1)}\lim\sup\left\Vert v_{\varepsilon}\right\Vert
_{L^{\infty}(\mathcal{K})}=\eta\tau
\]
for any $\eta,$ then $\lim\left\Vert u_{\varepsilon}\right\Vert _{L^{\infty
}(\mathcal{K})}=0.$
\end{proof}

\section{The subcritical case $1<q<q_{\ast}$\label{vs}}

We first recall the following results of \cite[Theorem 3.2 and Proposition
5.1]{BeDa} for the Dirichlet problem.

\begin{theorem}
[\cite{BeDa}]\label{BeDa}Let $1<q<q_{\ast}$ and $\Omega$ be a smooth bounded
domain. Then for any $u_{0}\in\mathcal{M}_{b}(\Omega)$ and any $T\in\left(
0,\infty\right]  $ there exists a weak solution of problem $(D_{\Omega,\infty
})$ such that $u(.,0)=u_{0}$ in the weak sense of $\mathcal{M}_{b}(\Omega):$
\begin{equation}
\lim_{t\rightarrow0}\int_{\Omega}u(.,t)\varphi dx=\int_{\Omega}\varphi
du_{0},\qquad\forall\varphi\in C_{b}(\Omega), \label{cb}%
\end{equation}
and $u$ is given equivalently by the semi-group formula%
\begin{equation}
u(.,t)=e^{t\Delta}u_{0}-\int_{0}^{t}e^{(t-s)\Delta}\left\vert \nabla
u(.,s)\right\vert ^{q}(s)ds\qquad\text{in }L^{1}(\Omega), \label{seg}%
\end{equation}
where $e^{t\Delta}u_{0}$ is the unique weak solution $w$ of the heat equation
such that
\begin{equation}
\lim_{t\rightarrow0}\int_{\Omega}w(.,t)\varphi dx=\int_{\Omega}\varphi
du_{0},\qquad\forall\varphi\in C_{b}(\Omega). \label{limi}%
\end{equation}
Moreover $u\in C^{2,1}(Q_{\Omega,\infty})$, and $u\in C\left(  \overline
{Q_{\Omega,\epsilon,\infty}}\right)  $ for any $\epsilon>0.$ And $u$ is the
unique weak solution of problem $(D_{\Omega,T})$ for any $T\in\left(
0,\infty\right)  .$
\end{theorem}

This solution was obtained from the Banach fixed point theorem. The existence
was also obtained by approximation in \cite{Al}, from the pioneer results of
\cite{BoGa}. Here we give a shorter proof of Theorem \ref{BeDa} when $u_{0}$
is nonnegative, and firm in details the convergence\noindent:

\begin{proposition}
\label{conv} Suppose $1<q<q_{\ast}.$ Let $u_{0}\in\mathcal{M}_{b}^{+}%
(\Omega),$ and $\left(  u_{0,n}\right)  $ be any sequence of functions of
$C_{b}^{1}(\overline{\Omega})\cap C_{0}(\Omega)$ converging weak $^{\ast}$ to
$u_{0}$, such that $\left\Vert u_{0,n}\right\Vert _{L^{1}\left(
\Omega\right)  }\leqq\left\Vert u_{0}\right\Vert _{\mathcal{M}_{b}(\Omega)}.$
Let $u_{n}$ be the classical solution of $(D_{\Omega,\infty})$ with initial
data $u_{0,n}.\medskip$

Then $(u_{n})$ converges in $C_{loc}^{2,1}(Q_{\Omega,\infty})\cap
C_{loc}^{1,0}(\overline{\Omega}\times(0,\infty))$ to a function $u\in
L_{loc}^{q}(\left[  0,\infty\right)  ;W_{0}^{1,q}(\Omega))$ and $u$ is the
unique solution of $(D_{\Omega,T}),$ (\ref{cb}) for any $T>0.$ And $u$
satisfies the estimates (\ref{wa}) and (\ref{univ}).
\end{proposition}

\begin{proof}
There holds
\[
u_{n}(.,t)=e^{t\Delta}u_{0,n}-\int_{0}^{t}e^{(t-s)\Delta}\left\vert \nabla
u_{n}(.,s)\right\vert ^{q}(s)ds\qquad\text{in }L^{1}(\Omega).
\]
From estimate (\ref{wa}) and Theorem \ref{Dir}, since $q<2,$ one can extract a
subsequence, still denoted $\left(  u_{n}\right)  $, converging in
$C_{loc}^{2,1}(Q_{\Omega,\infty})\cap C_{loc}^{1}(\overline{\Omega}%
\times(0,\infty))$ to a weak solution $u$ of $(D_{\Omega,\infty})$. And
\begin{equation}
\int_{\Omega}u_{n}(.,t)dx+\int_{0}^{t}\int_{\Omega}\left\vert \nabla
u_{n}(.,s)\right\vert ^{q}(s)dxds-\int_{0}^{t}\int_{\partial\Omega}%
\frac{\partial u_{n}}{\partial\nu}(.,s)dxds=\int_{\Omega}u_{0,n}dx;
\label{fou}%
\end{equation}
hence $\left\vert \nabla u_{n}\right\vert ^{q}$ is bounded in $L^{1}%
(Q_{\Omega,\infty})$ by $\left\Vert u_{0}\right\Vert _{\mathcal{M}_{b}%
(\Omega)}.$ Then from \cite[Lemma 3.3]{BaPi}, $\left(  u_{n}\right)  $ is
bounded in $L^{\gamma}((0,\tau),W_{0}^{1,\gamma}(\Omega))$ for any $\gamma
\in\left[  1,q_{\ast}\right)  $. Thus $(\left\vert \nabla u_{n}\right\vert
^{q})$ converges to $\left\vert \nabla u\right\vert ^{q}$ in $L_{loc}%
^{1}(\left[  0,\infty\right)  ,L^{1}(\Omega)),$ and $\left(  e^{t\Delta
}u_{0,n}\right)  $ converges a.e. to $e^{t\Delta}u_{0}$, and $u$ satisfies
(\ref{seg}). Moreover $u$ is the unique solution of $(D_{\Omega,T})$. Indeed
let $v$ be any other solution; taking $\gamma\in\left(  q,q_{\ast}\right)  ,$
there holds from \cite[Lemma 3.3]{BaPi}, with constants $C=C(\gamma,\Omega),$
\begin{align*}
\left\Vert \nabla(u-v)\right\Vert _{L^{\gamma}(Q_{\Omega,\tau})}  &  \leqq
C\left\Vert \left\vert \nabla u\right\vert ^{q}-\left\vert \nabla v\right\vert
^{q}\right\Vert _{_{L^{1}(Q_{\Omega,\tau})}}\\
&  \leqq C(\left\Vert \nabla u\right\Vert _{L^{q}(Q_{\Omega,T})}%
^{q-1}+\left\Vert \nabla v\right\Vert _{L^{q}(Q_{\Omega,T})}^{q-1})\left\Vert
\nabla(u-v)\right\Vert _{L^{q}(Q_{\Omega,\tau})}\\
&  \leqq C\left\Vert u_{0}\right\Vert _{\mathcal{M}_{b}(\Omega)}\left\Vert
\nabla(u-v)\right\Vert _{L^{\gamma}(Q_{\Omega,\tau})}\tau^{\frac{\gamma
-q}{\gamma q}},
\end{align*}
hence $v=u$ on $(0,\tau)$ for $\tau\leqq C=C(\gamma,\Omega,u_{0}),$ and then
on $(0,T).$ Then the whole sequence $(u_{n})$ converges to $u.$
\end{proof}

\begin{remark}
\label{regi} Applying Proposition \ref{conv} on $\left(  \epsilon,T\right)  $
for $\epsilon>0,$ we deduce regularity results: any weak solution $u$ of
$(D_{\Omega,T})$ extends as a solution of the problem $(D_{\Omega,\infty}),$
and $u\in C^{2,1}(Q_{\Omega,\infty})$, and $u\in C\left(  \overline
{Q_{\Omega,\epsilon,\infty}}\right)  $ for any $\epsilon>0,$ and $u$ satisfies
the universal estimates (\ref{wa}) and (\ref{univ}). In turn $u\in
C_{loc}^{1,0}(Q_{\Omega,\infty})$ from Theorem \ref{Dir}.
\end{remark}

\begin{notation}
For any $k>0$, we denote by $u^{k,\Omega}$ the above solution of
$(D_{\Omega,\infty})$ with initial data $k\delta_{0}$.
\end{notation}

\subsection{The case $\Omega=$ $\mathbb{R}^{N}$\label{more}}

We first show that the function $Y$ constructed at Proposition \ref{L.3.4} is
a VSS:

\begin{lemma}
\label{Y} The function $Y$ is a maximal V.S.S. in $Q_{\mathbb{R}^{N},\infty},$
and coincides with the radial self-similar solution constructed in \cite{QW}.
It satisfies
\begin{equation}
\lim_{t\rightarrow0}\int_{\mathbb{R}^{N}\backslash B_{r}}Y(.,t)dx=0,\qquad
\forall r>0.\label{vive}%
\end{equation}

\end{lemma}

\begin{proof}
Consider any ball $B_{p}$ with $p\geqq1.$ We can approximate the function
$u^{k,B_{p}}$ by $u_{\varepsilon}^{k,B_{p}}$\noindent, solution with initial
data $k\rho_{\varepsilon},$ where $\left(  \rho_{\varepsilon}\right)  $ is a
sequence of mollifiers with support in $B_{\varepsilon}\subset B_{1}.$ For any
$\eta\in\left(  0,1\right)  ,$ there holds $u_{\varepsilon}^{k,B_{p}}\leqq
Y_{\eta}$ for $\varepsilon<\eta.$ Then we find $u^{k,B_{p}}\leqq Y.$ As a
first consequence, $Y\neq0,\ $and for any ball $B_{r}$ such that $r<1,$ taking
$\varphi\in C_{c}(B_{r})$ with values in $\left[  0,1\right]  ,$ such that
$\varphi\equiv1$ on $B_{r/2},$
\[
\underline{\lim}_{t\rightarrow0}\int_{B_{r}}Y(.,t)dx\geqq\lim_{t\rightarrow
0}\int_{B_{r}}u^{k,B_{p}}(.,t)\varphi dx=k,\text{ }%
\]
thus $Y$ satisfies (\ref{VS}). From (\ref{fil}), $Y$ is the unique radial
self-similar VSS constructed in \ref{QW}. It satisfies (\ref{vive}), since
$Y(x,t)=t^{-a/2}f(t^{-1/2}\left\vert x\right\vert ),$ and $\lim_{r\rightarrow
\infty}r^{a-N}e^{r^{2}/4}f(r)>0,$ from \cite[Theorem 2.1]{QW}, which implies
(\ref{RR}). And $Y$ is a maximal VSS, since $Y$ is greater than any weak
solution of (\ref{un}), (\ref{R}), from Proposition \ref{muji}.\medskip
\end{proof}

In \cite{BeLaEx}, a VSS $U$ is constructed as the limit of a sequence of
solutions $u^{k}$ of (\ref{un}) in $Q_{\mathbb{R}^{N},\infty}$ with initial
data $k\delta_{0}$, constructed in \cite{BeLa99}. The proof is based on
difficult estimates of the gradient obtained from from the Bernstein technique
by derivation of equation, showing that $U$ satisfies (\ref{vi}), (\ref{plic})
and (\ref{ploc}); and is minimal in that class, from \cite[ Theorem
3.8]{BeKoLa}. Here we prove again the existence of the $u^{k}$ and $U$ in a
very simple way:

\begin{lemma}
\label{solk}(i) For any $k>0$ there exists a weak solution $u^{k}$ of
(\ref{un}) in $Q_{\mathbb{R}^{N},\infty}$, such that $u^{k}\in L^{\infty
}((0,\infty);L^{1}(\mathbb{R}^{N}))$ and $\left\vert \nabla u^{k}\right\vert
\in L^{q}(Q_{\mathbb{R}^{N},\infty}),$ with initial data $k\delta_{0},$ in the
weak sense of $\mathcal{M}_{b}\left(  \mathbb{R}^{N}\right)  $
\begin{equation}
\lim_{t\rightarrow0}\int_{\mathbb{R}^{N}}u^{k}(.,t)\psi dx=k\psi
(0),\qquad\forall\psi\in C_{b}(\mathbb{R}^{N});\label{huc}%
\end{equation}
and $u^{k}=\sup u^{k,B_{p}}$, where $u^{k,B_{p}}$ is the solution of the
Dirichlet problem $(D_{B_{p},\infty})$ with initial data $k\delta_{0}.$

(ii) As $k\rightarrow\infty,$ $u^{k}$ converges in $C_{loc}^{2,1}%
(Q_{\mathbb{R}^{N},\infty})$ to a V.S.S $U$ in $Q_{\mathbb{R}^{N},\infty}.$
\end{lemma}

\begin{proof}
(i) Let $k>0$ be fixed. Consider again the sequence $\left(  u^{k,B_{p}%
}\right)  $\noindent. We have
\begin{equation}
u^{k,B_{p}}(.,t)\leqq Y(.,t)\leqq C(1+t^{-\frac{1}{q-1}}).\label{mes}%
\end{equation}
from Proposition \ref{L.3.4}. From Theorem \ref{T.2.1} the sequence converges
in $C_{loc}^{2,1}(Q_{\Omega,\infty})$ to a solution $u^{k}$ of equation
(\ref{un}) in $Q_{\mathbb{R}^{N},\infty},$ and $u^{k}\leqq Y,$ thus $u^{k}$
satisfies (\ref{RR}) from (\ref{dir}). Moreover for any $t>0,$ from
(\ref{seg}) and (\ref{limi}),
\[
\int_{B_{p}}u^{k,B_{p}}(.,t)dx\leqq k,\qquad\lim_{t\rightarrow0}\int_{B_{p}%
}u^{k,B_{p}}(.,t)dx=k.
\]
Then from the Fatou Lemma,
\[
\int_{\mathbb{R}^{N}}u^{k}(.,t)dx\leqq k.
\]
In turn from Proposition \ref{dic}, $u^{k}(.,t)$ converges weak$^{\ast}$ to a
Radon measure $\mu$, concentrated at $0,$ then $\mu=k^{\prime}\delta_{0},$
$k^{\prime}>0.$ Otherwise $u^{k,B_{p}}\leqq u^{k},$ then $\int_{B_{p}%
}u^{k,B_{p}}(.,t)dx\leqq\int_{\mathbb{R}^{N}}u^{k}(.,t)dx,$ thus
\[
k\leqq\lim\inf_{t\rightarrow0}\int_{\mathbb{R}^{N}}u^{k}(.,t)dx;
\]
then $\lim_{t\rightarrow0}\int_{\mathbb{R}^{N}}u^{k}(.,t)dx=k.$ Taking
$\varphi_{p}\in\mathcal{D}^{+}(\mathbb{R}^{N}),$ with values in $\left[
0,1\right]  ,$ such that $\varphi_{p}=1$ on $B_{p},$ we get
\[
\int_{B_{p}}u^{k,B_{p}}(.,t)dx\leqq\int_{\mathbb{R}^{N}}u^{k}(.,t)\varphi
_{p}dx\leqq\int_{\mathbb{R}^{N}}u^{k}(.,t)dx
\]
hence $k^{\prime}=k;$ thus $u^{k}(.,t)$ converges weak $^{\ast}$ to
$k\delta_{0}$ as $t\rightarrow0$. In fact the convergence holds in the weak
sense of $\mathcal{M}_{b}(\mathbb{R}^{N}).$ Indeed for any $\psi\in C_{b}%
^{+}(\mathbb{R}^{N})$, using a function $\varphi\in C_{c}(\mathbb{R}^{N})$
with values in $\left[  0,1\right]  $ such that $\varphi\equiv1$ on a ball
$B_{r},$ we can write
\[
\int_{\mathbb{R}^{N}}u^{k}(.,t)\psi dx=\int_{\mathbb{R}^{N}}u^{k}%
(.,t)\psi\varphi dx+\int_{\mathbb{R}^{N}}u^{k}(.,t)\psi(1-\varphi)dx,
\]
and
\[
\int_{\mathbb{R}^{N}}u^{k}(.,t)\psi(1-\varphi)dx\leqq\left\Vert \psi
\right\Vert _{L^{\infty}(\mathbb{R}^{N})}\int_{\mathbb{R}^{N}\backslash B_{r}%
}u^{k}(.,t)dx\leqq\left\Vert \psi\right\Vert _{L^{\infty}(\Omega)}%
\int_{\mathbb{R}^{N}\backslash B_{r}}Y(.,t)dx
\]
and the right hand side tends to 0 from (\ref{vive}). From (\ref{fou}), we
find
\[
\left\Vert \left\vert \nabla u_{\varepsilon}^{k,B_{p}}\right\vert
^{q}\right\Vert _{L^{1}(Q_{B_{p},\infty})}\leqq k\left\Vert \rho_{\varepsilon
}\right\Vert _{L^{1}(B_{p})}=k,
\]
hence $\left\Vert \left\vert \nabla u^{k,B_{p}}\right\vert ^{q}\right\Vert
_{L^{1}(Q_{B_{p},\infty})}\leqq k,$ and finally $\left\Vert \left\vert \nabla
u^{k}\right\vert ^{q}\right\Vert _{L^{1}(Q_{\mathbb{R}^{N},\infty})}\leqq k,$
from the convergence a.e. of the gradients.

(ii) From (\ref{mes}) or from Proposition (\ref{muji}), there holds
\begin{equation}
u^{k}(.,t)\leqq Y(.,t)\leqq C(1+t^{-\frac{1}{q-1}}). \label{mem}%
\end{equation}
From Theorem \ref{T.2.1}, $u^{k}$ converges in $C_{loc}^{2,1}(Q_{\mathbb{R}%
^{N},\infty})$ to a weak solution $U$ of equation (\ref{un}). Then $u^{k}\leqq
U\leqq Y,$ thus $U$ satisfies (\ref{VS}) and (\ref{vive}) as $Y$. Hence $U$ is
a VSS in $Q_{\mathbb{R}^{N},\infty}.\medskip$
\end{proof}

Next we prove the uniqueness of the VSS:$\medskip$

\begin{proof}
[Proof of Theorem \ref{unique}]Let us show that $U$ is minimal VSS. Let $u$ be
any VSS in $Q_{\mathbb{R}^{N},\infty}.$ From Proposition \ref{muji}, and
(\ref{vir}), $u\in C^{2,1}(Q_{\mathbb{R}^{N},\infty})\cap C((0,\infty
);C_{b}^{2}(\mathbb{R}^{N}))$ and $u\leqq Y$. For fixed $k>0$ and $p>1,$ one
constructs a sequence of functions $u_{0,n}^{k}\in\mathcal{D}^{+}%
(\mathbb{R}^{N})$ with support in \ $B_{1}$ such that
\[
u_{0,n}^{k}\leqq u(.,\frac{1}{n})\quad\text{in }\mathbb{R}^{N},\qquad
\lim_{n\rightarrow\infty}\int_{\mathbb{R}^{N}}u_{0,n}^{k}dx=k.
\]
Indeed $\left\Vert u(.,1/n)\right\Vert _{{L^{1}(}\mathbb{R}^{N})}$ tends to
$\infty,$ then, for $n$ large enough, there exists $s_{n,k}>0$ such that
$\left\Vert T_{s_{n,k}}(u)(.,1/n)\right\Vert _{{L^{1}(}\mathbb{R}^{N})}=k.$
And $\varepsilon_{n}=\left\Vert u(.,1/n)\right\Vert _{{L^{1}(}\mathbb{R}%
^{N}\backslash B_{1})}+\left\Vert u(.,1/n)\right\Vert _{{L^{\infty}%
(}\mathbb{R}^{N}\backslash B_{1})}$ tends to $0,$ from (\ref{vive}) and
(\ref{dir}). Then $v_{n}^{k}=(T_{s_{n,k}}(u)(.,1/n)-2\varepsilon_{n})^{+}$ has
a compact support in $B_{1},$ and we can take for $u_{0,n}^{k}$ a suitable
regularization of $v_{n}^{k}$. Let us call $u_{n}^{k,B_{p}}$ the solution of
$(D_{B_{p},\infty})$ with initial data $u_{0,n}^{k}.$ Then we obtain that
$u_{n}^{k,B_{p}}(.,t)\leqq u(.,t+1/n)$ from the comparison principle. As
$n\rightarrow\infty,$ $u_{0,n}^{k}$ converges to $k\delta_{0}$ weakly in
$\mathcal{M}_{b}(B_{p})$, since for any $\psi\in C_{b}^{+}(B_{p}),$ and any
$r\in\left(  0,1\right)  ,$%
\begin{align*}
\left\vert \int_{B_{p}}u_{0,n}^{k}\psi dx-k\psi(0)\right\vert  &  \leqq
\psi(0)\left\vert \int_{B_{p}}(u_{0,n}^{k}-k)dx\right\vert \\
&  +2\left\Vert \psi\right\Vert _{{L^{\infty}(}\mathbb{B}_{p})}\int
_{\mathbb{R}^{N}\backslash B_{r}}u(.,\frac{1}{n})dx+\sup_{B_{r}}\left\vert
\psi-\psi(0)\right\vert \int_{\mathbb{R}^{N}}u_{0,n}^{k}dx.
\end{align*}
Then $u_{n}^{k,B_{p}}$ converges to $u^{k,B_{p}}$ from Proposition \ref{conv},
and $u^{k,B_{p}}\leqq u.$ From Lemma \ref{solk}, we get $u^{k}\leqq u\leqq Y.$
As $k\rightarrow\infty,$ we deduce that $U\leqq u\leqq Y.$ Moreover $U$ is
radial and self-similar, then $U=Y=u$ from \cite{QW}. \medskip
\end{proof}

Finally we describe all the solutions:\medskip

\begin{proof}
[Proof of Theorem \ref{desc}]Let $u$ be any weak solution of (\ref{un}),
(\ref{RR}). Either (\ref{VS}) holds, then $u=Y.$ Or there exists a ball
$B_{r}$ such that $\int_{B_{r}}u(.,t)dx$ stays bounded as $t\rightarrow0.$
Then $u\in L_{loc}^{\infty}{(}\left[  0,T\right)  {;L_{loc}^{1}(}%
\mathbb{R}^{N})),$ from Corollary \ref{sous}. From Proposition \ref{dic},
$u(.,t)$ converges weak$^{\ast}$ to a measure $\mu$ as $t\rightarrow0.$ Then
$\mu$ is concentrated at $0$ from (\ref{RR}), hence the exists $k\geqq0$ such
that $\mu=k\delta_{0},$ and (\ref{ink}) holds as in Lemma \ref{solk}, since
$u\leqq Y$. If $k=0,$ then $u\equiv0$ from Theorem \ref{comp}.\medskip

Next we show the uniqueness, namely that $u=$ $u^{k}$ constructed at Lemma
\ref{solk}. \textit{Here only} we use the gradient estimates obtained by the
Bernstein technique. We have $u\in C((0,\infty);C_{b}^{2}(\mathbb{R}^{N}))$
from Proposition (\ref{muji}), and $u\in L^{\infty}((0,\infty);L^{1}%
(\mathbb{R}^{N}))$ from (\ref{gam}) or (\ref{vive}) thus $u\in C((0,\infty
);L^{1}(\mathbb{R}^{N})).$ From \cite{BeLa99}, \cite{BeBALa}, for any
$\epsilon>0,$ and any $t\geqq\epsilon,$ we have the semi-group formula
\begin{equation}
u(.,t)=e^{(t-\epsilon)\Delta}u(.,\epsilon)-\int_{\epsilon}^{t}e^{(t-s)\Delta
}\left\vert \nabla u\right\vert ^{q}(s)ds\qquad\text{in }L^{1}(\mathbb{R}%
^{N}), \label{true}%
\end{equation}
and there exists $C(q)$ such that for any $t>0,$
\[
\left\vert \nabla u(.,t)\right\vert ^{q}\leqq C(q)(t-\epsilon)^{-1}u(.,t).
\]
Going to the limit as $\epsilon\rightarrow0$ we deduce from (\ref{ploc}),
since $u\leqq Y,$
\[
\left\Vert \nabla u(.,t)\right\Vert _{L^{\infty}(\mathbb{R}^{N})}\leqq
C(q)t^{-1/q}\left\Vert Y(.,t)\right\Vert _{L^{\infty}(\mathbb{R}^{N})}%
^{1/q}\leqq Ct^{-(N+2)/2q}%
\]
where $C=C(N,q).$ From (\ref{ink}) and (\ref{true}) there holds $\left\vert
\nabla u\right\vert ^{q}\in L_{loc}^{1}(\left[  0,\infty\right)
;L^{1}(\mathbb{R}^{N})).$ Otherwise $e^{(t-\epsilon)\Delta}u(x,\epsilon)$
converges to $kg$ in $\mathcal{C}_{b}^{\prime}(\mathbb{R}^{N}),$ where $g$ is
the heat kernel, then
\[
u(.,t)=kg-\int_{0}^{t}e^{(t-s)\Delta}\left\vert \nabla u\right\vert
^{q}(s)ds\qquad\text{in }\mathcal{C}_{b}^{\prime}(\mathbb{R}^{N}).
\]
Then%
\[
(u-u^{k})(.,t)=-\int_{0}^{t}e^{(t-s)\Delta}(\left\vert \nabla u\right\vert
^{q}-\left\vert \nabla u^{k}\right\vert ^{q})(s)ds\qquad\text{in }%
L^{1}(\mathbb{R}^{N}),
\]
\begin{align*}
\left\Vert \nabla(u-u^{k})(.,t)\right\Vert _{L^{q}\left(  \mathbb{R}%
^{N}\right)  }  &  \leqq\int_{0}^{t}\left\Vert e^{(t-s)\Delta}\right\Vert
_{L^{1}\left(  \mathbb{R}^{N}\right)  }\left\Vert \left\vert \nabla
u(.,s)\right\vert ^{q}-\left\vert \nabla u^{k}(.,s)\right\vert ^{q}\right\Vert
_{L^{q}\left(  \mathbb{R}^{N}\right)  }ds\\
&  \leqq C\int_{0}^{t}(t-s)^{-1/2}s^{-(q-1)(N+2)/2q}\left\Vert \nabla
(u-u^{k})(.,s)\right\Vert _{L^{q}\left(  \mathbb{R}^{N}\right)  }ds.
\end{align*}
Thus $\nabla(u-u^{k})(.,t)=0$ in $L^{q}\left(  \mathbb{R}^{N}\right)  ,$ from
the singular Gronwall lemma, valid since $q<\frac{N+2}{N+1};$ hence $u=u^{k}.$
\end{proof}

\begin{remark}
This uniqueness result is a special case of a general one given for measure
data in \cite[Theorem 3.27]{BiDao}.
\end{remark}

\subsection{The Dirichlet problem $(D_{\Omega,\infty})$ \label{end}}

Here $\Omega$ is bounded, and we consider the weak solutions of the problem
$(D_{\Omega,\infty})$ such that
\begin{equation}
\lim_{t\rightarrow0}\int_{\Omega}u(.,t)\varphi dx=0,\qquad\forall\varphi\in
C_{c}(\overline{\Omega}\backslash\left\{  0\right\}  ). \label{Ra}%
\end{equation}

\noindent First, we give regularity properties of these solutions.

\begin{lemma}
\label{dru}Any weak solution $u$ of $(D_{\Omega,\infty}),$ (\ref{Ra}), in
$Q_{\Omega,\infty}$ satisfies
\[
u\in C^{1,0}(\overline{\Omega}\backslash\left\{  0\right\}  \times\left[
0,\infty\right)  )\cap C^{1,0}(\overline{\Omega}\times(0,\infty))\cap
C^{2,1}\left(  Q_{\Omega,\infty}\right)  .
\]

\end{lemma}

\begin{proof}
We know that $u\in C^{1,0}(\overline{\Omega}\times(0,\infty))\cap
C^{2,1}\left(  Q_{\Omega,\infty}\right)  ,$ see Remark \ref{regi}. Moreover
$u\in C^{2,1}(\Omega_{0}\times\left[  0,\infty\right)  )$ and
$u(x,0)=0,\;\forall x\in\Omega_{0},$ from Corollary \ref{sous}. Let $B_{\eta
}\subset\subset\Omega$ be fixed, and $\Omega_{\eta}=\Omega\backslash
\overline{B_{\eta}}.$ Then $u\in C^{1}\left(  \partial B_{\eta}\times\left[
0,\infty\right)  \right)  ,$ thus for any $T\in\left(  0,\infty\right)  ,$
there exists $C_{\tau}>0$ such that $u(.,t)\leqq C_{\tau}t$ on $\partial
B_{\eta}\times\left[  0,T\right)  .$ Then the function $w=u-C_{\tau}t$ solves
\[
w_{t}-\Delta w=-\left\vert \nabla u\right\vert ^{q}-C_{\tau}\quad\quad\text{in
}\mathcal{D}^{\prime}\left(  Q_{\Omega_{\eta},T}\right)  ,
\]
then $w^{+}\in C((0,T);L^{1}\left(  \Omega_{\eta}\right)  \cap L_{loc}%
^{1}((0,T);W_{0}^{1,1}\left(  \Omega_{\eta}\right)  ),$ and
\[
w_{t}^{+}-\Delta w^{+}\leqq0\quad\quad\text{in }\mathcal{D}^{\prime}\left(
Q_{\Omega_{\eta},T}\right)
\]
from the Kato inequality. Moreover, from assumption (\ref{Ra}), $w^{+}\in
L^{\infty}((0,T);L^{1}\left(  \Omega_{\eta}\right)  )$ and $w^{+}(.,t)$
converges to $0$ in the weak sense of $\mathcal{M}_{b}\left(  \Omega_{\eta
}\right)  .$ As a consequence, $w\leqq0,$ from \cite[Lemma 3.4]{BaPi}; thus
$u(.,t)\leqq C_{T}t$ in $\Omega_{\eta,T}.$ Then the function $\overline{u}$
defined by (\ref{uba}) is bounded in $Q_{\Omega_{\eta},\tau}.$ Hence
$\overline{u}\in C^{1,0}(\overline{\Omega_{\eta}}\times(-T,T))$ from Theorem
\ref{Dir}, thus $u\in C^{1,0}(\overline{\Omega}\backslash\left\{  0\right\}
\times\left[  0,\infty\right)  ).$
\end{proof}

\begin{definition}
Let $T\in\left(  0,\infty\right]  .$ We call VSS in $Q_{\Omega,T}$ any weak
solution $u$ of the Dirichlet problem $(D_{\Omega,T})$, (\ref{Ra}), such that
\begin{equation}
\lim_{t\rightarrow0}\int_{B_{r}}u(.,t)dx=\infty,\qquad\forall B_{r}%
\subset\Omega. \label{Sb}%
\end{equation}

\end{definition}

\begin{remark}
From Remark \ref{regi}, any VSS in $Q_{\Omega,T}$ extends as a VSS in
$Q_{\Omega,\infty},$ and satisfies (\ref{wa}) and (\ref{univ}).
\end{remark}

Next we prove the existence and uniqueness of the VSS. Our proof is based on
the uniqueness of the VSS in $\mathbb{R}^{N},$ and does not use the uniqueness
of the function $u^{k}.\medskip$

\begin{proof}
[Proof of Theorem \ref{vss}]\textit{(i) Existence of a minimal VSS. }For any
$k>0$ we consider the solution $u^{k,\Omega}$ of $(D_{\Omega,\infty})$ with
initial data $k\delta_{0}$. By regularization as in Lemma \ref{solk}, we
obtain that $u^{k,\Omega}\leqq Y.$ The sequence $\left(  u^{k,\Omega}\right)
$ is nondecreasing. From estimate (\ref{wa}) and Theorem \ref{Dir}\noindent,
$\left(  u^{k,\Omega}\right)  $ converges in $C_{loc}^{2,1}(Q_{\Omega,\infty
})\cap C_{loc}^{1,0}(\overline{\Omega}\times(0,\infty))$ to a weak solution
$U^{\Omega}$ of $(D_{\Omega,\infty})$, and then $U^{\Omega}\leqq Y.$ Hence
$U^{\Omega}$ satisfies (\ref{Sb}), and (\ref{Ra}) from (\ref{vive}), thus
$U^{\Omega}$ is a VSS in $\Omega$. Next we show that $U^{\Omega}$ is minimal.
Consider any VSS $u$ in $Q_{\Omega,\infty}.$ Let $k>0$ be fixed. As in the
proof of Theorem \ref{unique}, one constructs a sequence $u_{n}^{k,\Omega}$of
solutions of $(D_{\Omega,\infty})$ with initial data functions $u_{0,n}%
^{k,\Omega}\in\mathcal{D}(\Omega)$ such that
\[
0\leqq u_{0,n}^{k,\Omega}\leqq u(.,\frac{1}{n})\quad\text{in }\Omega
,\qquad\lim_{n\rightarrow\infty}\int_{\Omega}u_{0,n}^{k,\Omega}dx=k.
\]
We still find $u_{n,p}^{k}(.,t)\leqq u(.,t+1/n)$ from the comparison
principle, valid from Lemma \ref{dru}. As $n\rightarrow\infty,$ $u_{0,n}%
^{k,\Omega}$ converges to $k\delta_{0}$ weakly in $\mathcal{M}_{b}(\Omega)$,
then $u_{n}^{k,\Omega}$ converges to $u^{k,\Omega}$ from Proposition
\ref{conv}. Then $u^{k,\Omega}\leqq u$ for any $k>0,$ thus $U^{\Omega}\leqq
u.$\medskip

\textit{(ii) Existence of a maximal VSS. }For any ball $B_{\eta}\subset
\subset\Omega$, we consider the function $Y_{\eta}^{\Omega}$ defined at
Theorem \ref{L.3.3}. Consider again any VSS $u$ in $\Omega,$ and follow the
proof of Proposition \ref{muji}, replacing $B_{r}$ by $\Omega.$ Let
$\varepsilon>0$ be fixed. From Lemma \ref{dru}, for any ball $B_{\eta}%
\subset\subset\Omega,$ setting $\Omega_{\eta}=\Omega\backslash\overline
{B_{\eta}}$ there is $\delta_{\eta}>0$ such that
\begin{equation}
u(x,t)<\varepsilon,\qquad\text{in }Q_{\Omega_{\eta},\delta_{\eta}}\text{ }%
\end{equation}
Next, for any $\delta\in(0,\delta_{\eta})$, from the comparison principle in
$Q_{\Omega,\delta,\tau}$ we deduce that
\[
u(x,t)\leqq Y_{2\eta}^{\Omega}(x,t-\delta)+\varepsilon\qquad\text{in
}Q_{\Omega,\delta,\tau}.
\]
As $\delta$ tends to $0$, and then $\varepsilon\rightarrow0$, we deduce that
$u\leqq Y_{2\eta}^{\Omega}$ in $Q_{\Omega,\infty}.$ We observe that $Y_{\eta
}^{\Omega}\leqq Y_{\eta^{\prime}}^{\Omega}$ for any $\eta\leqq\eta^{\prime}.$
From the estimates (\ref{wa}) and Theorem \ref{T.2.1}, $Y_{\eta}^{\Omega}$
converges in $C_{loc}^{1,0}(\overline{\Omega}\times(0,\infty))$ to a classical
solution $Y^{\Omega}$ of $(D_{\Omega,\infty}),$ and $u\leqq Y^{\Omega}$.
Moreover $Y^{\Omega}$ satisfies (\ref{Sb}), since $Y^{\Omega}\geqq U,$ and
(\ref{Ra}) since $Y^{\Omega}\leqq Y,$ then $Y^{\Omega}$ is a maximal VSS in
$\Omega.$\medskip

\textit{(iii) Uniqueness. }For fixed $k>0,$ we intend to compare $u^{k,\Omega
}$ with $u^{k},$ by approximation. Let $0<\eta<r$ be fixed such that
$B_{r}\subset\subset\Omega.$ Consider again the function $Y_{\eta}$ defined by
(\ref{py}). Let $\delta>0$ be fixed. From (\ref{fil}), there exists
$\tau_{\delta}>0$ such that $\sup_{(\mathbb{R}^{N}\backslash B_{r}%
)\times\left[  0,\tau_{\delta}\right]  }Y_{\eta}\leqq\delta.$ Let $\left(
\rho_{\varepsilon}\right)  $ be a sequence of mollifiers with support in
$B_{\varepsilon}\subset B_{\eta}$. Let $u_{\varepsilon}^{k,\Omega}$ be the
solution of $(D_{\Omega,\infty})$ in $Q_{\Omega,\infty}$ with initial data
$k\rho_{\varepsilon}.$ For any $p>1$ such that $\Omega\subset B_{p},$ let
$u_{\varepsilon}^{k,B_{p}}$ be the solution of $(D_{B_{p},\infty})$ with the
same initial data. By definition of $Y_{\eta}^{B_{p}}$ and $Y_{\eta},$ there
holds $u_{\varepsilon}^{k,B_{p}}\leqq Y_{\eta}^{B_{p}}\leqq Y_{\eta},$ hence
$\sup_{\partial\Omega\times\left[  0,\tau_{\delta}\right]  }$ $u_{\varepsilon
}^{k,B_{p}}\leqq\delta.$ Applying the comparison principle to the smooth
functions $u_{\varepsilon}^{k,\Omega}$ and $u_{\varepsilon}^{k,B_{p}}$ in
$\overline{\Omega}\times\left[  0,\infty\right)  ,$ we obtain that
\[
u_{\varepsilon}^{k,B_{p}}\leqq u_{\varepsilon}^{k,\Omega}+\delta\qquad\text{in
}\overline{\Omega}\times\left[  0,\tau_{\delta}\right]  .
\]
Going to the limit as $\varepsilon\rightarrow0$ from Proposition \ref{conv}
and then as $p\rightarrow\infty$ from Lemma \ref{solk}, we obtain that
\[
u^{k}\leqq u^{k,\Omega}+\delta\qquad\text{in }\overline{\Omega}\times\left(
0,\tau_{\delta}\right]  ;
\]
and going to the limit as $k\rightarrow\infty,$ we find
\[
U\leqq U^{\Omega}+\delta\qquad\text{in }\overline{\Omega}\times\left(
0,\tau_{\delta}\right]  .
\]
The function $W^{\Omega}=Y^{\Omega}-U^{\Omega}\in C^{1,0}(\overline{\Omega
}\backslash\left\{  0\right\}  \times\left[  0,\infty\right)  )\cap
C^{1,0}(\overline{\Omega}\times(0,\infty))$ from Lemma (\ref{dru}), and
$W^{\Omega}=0$ on $\partial\Omega\times\left[  0,\infty\right)  $. Since
$Y^{\Omega}\leqq Y=U,$ then $W^{\Omega}\leqq\delta$ in $\overline{\Omega
}\times\left(  0,\tau_{\delta}\right]  .$ Thus $W^{\Omega}(.,t)$ converges
uniformly to $0$ as $t\rightarrow0.$ Then for any $\varepsilon>0,$ $W^{\Omega
}-\varepsilon$ cannot have an extremal point in $Q_{\Omega,\infty},$ thus
$W^{\Omega}\leqq\varepsilon,$ hence $Y^{\Omega}=U^{\Omega}.\medskip$
\end{proof}

Finally we describe all the solutions as in the case of $\mathbb{R}^{N}:$

\begin{theorem}
Let $u$ be any weak solution of $(D_{\Omega,\infty}),$ (\ref{Ra}). Then either
$u=U^{\Omega}$, or there exists $k>0$ such that $u=u^{k,\Omega},$ or
$u\equiv0.$
\end{theorem}

\begin{proof}
Either $u=Y^{\Omega}$, or there exists a ball $B_{r}$ such that $\int_{B_{r}%
}u(.,t)dx$ stays bounded as $t\rightarrow0.$ Then from (\ref{Ra}), $u\in
L_{loc}^{\infty}{(}\left[  0,\infty\right)  {;L^{1}(\Omega}))$. From
Proposition \ref{dic}, $u(.,t)$ converges weak$^{\ast}$ to a measure $\mu$ as
$t\rightarrow0,$ concentrated at $\left\{  0\right\}  $ from (\ref{Ra}). Hence
the exists $k\geqq0$ such that $\mu=k\delta_{0},$ thus
\[
\lim_{t\rightarrow0}\int_{\Omega}u(.,t)\varphi dx=k\varphi(.,0),\qquad
\forall\varphi\in C_{c}(\Omega),
\]
and it holds for any $\varphi\in C_{b}(\Omega),$ from (\ref{Ra}). If $k>0,$
then $u=u^{k,\Omega}$ from uniqueness, see Proposition \ref{conv}. If $k=0,$
then $u\equiv0$ from Theorem \ref{comp}.
\end{proof}


\begin{thebibliography}{99}                                                                                               %


\bibitem {Al}M. Alaa, \textit{Solutions faibles d'\'{e}quations paraboliques
quasilin\'{e}aires avec donn\'{e}es mesures}, Ann. Math. Blaise Pascal, 3
(1996), 1-15.

\bibitem {AmBeA}L.Amour and M. Ben-Artzi, \textit{Global existence and decay
for Viscous Hamilton-Jacobi equations,} Nonlinear Analysis, Methods and
Applications, 31 (1998), 621-628.

\bibitem {AMST}F.\ Andreu, J. Mazon, S. Segura de Leon and J. Toledo,
\textit{Existence and uniqueness for a degenerate parabolic equation with
}$l^{1}$\textit{ data}, Trans. Amer. Math. Soc. 351 (1999), 285-306.

\bibitem {AS}D. G. Aronson and J. Serrin, \textit{Local behavior of solutions
of quasilinear parabolic equations}, Arch. Rat. Mech. Anal. 25 (1967), 81-122.

\bibitem {BaDLi}G. Barles and F. Da Lio, \textit{On generalized Dirichlet
problem for viscous Hamilton-Jacobie quations,} J. Maths Pures Appl. 83
(2004), 53-75.

\bibitem {BaPi}P. Baras and M. Pierre, \textit{Problemes paraboliques
semi-lineaires avec donn\'{e}es mesures,} Applicable Anal. 18 (1984),111-149.

\bibitem {BASoWe}M. Ben Artzi, P. Souplet and F. Weissler, \textit{The local
theory for Viscous Hamilton-Jacobi equations in Lebesgue spaces}, J. Math.
Pures Appl. 81 (2002), 343-378.

\bibitem {BeDa}S. Benachour, S. Dabuleanu, \textit{The mixed Cauchy-Dirichlet
problem for a viscous Hamilton-Jacobi equation}, Adv. Diff. Equ. 8 (2003), 1409-1452.

\bibitem {BeBALa}S. Benachour, M. Ben Artzi, and P. Lauren\c{c}ot,
\textit{Sharp decay estimates and vanishing viscosity for diffusive
Hamilton-Jacobi equations,} Adv. Differential Equations 14 (2009), 1--25.

\bibitem {BeLa99}S. Benachour and P. Lauren\c{c}ot, \textit{Global solutions
to viscous Hamilton-Jacobi equations with irregular initial data,} Comm.
Partial Differential Equations 24 (1999), 1999-2021.

\bibitem {BeLaEx}S. Benachour and P. Lauren\c{c}ot, \textit{Very singular
solutions to a nonlinear parabolic equation with absorption, I- Existence,
}Proc. Roy. Soc. Edinburgh Sect. A 131 (2001), 27-44.

\bibitem {BeKoLa}S. Benachour, H.Koch, and P. Lauren\c{c}ot, \textit{Very
singular solutions to a nonlinear parabolic equation with absorption, }II-
Uniqueness, Proc. Roy. Soc. Edinburgh Sect. A 134 (2004), 39-54.

\bibitem {BiChVe}M.F. Bidaut-V\'{e}ron, E. Chasseigne, and L. V\'{e}ron,
\textit{Initial trace of solutions of some quasilinear parabolic equations
with absorption,} J. Funct. Anal. 193 (2002) 140-205.

\bibitem {BiDao}M.F. Bidaut-V\'{e}ron, and A. N. Dao, \textit{Decay estimates
for parabolic equations with gradient terms, }preprint\textit{.}

\bibitem {BoGa}L. Boccardo and T. Gallouet, \textit{Nonlinear elliptic and
parabolic equations involving measure data,} J. Funct. Anal. 87 (1989), 149-169.

\bibitem {BrFr}H. Brezis and A. Friedman, \textit{Nonlinear parabolic
equations involving measures as initial conditions,} J.Math.Pures Appl. 62
(1983), 73-97.

\bibitem {BrPT}H. Brezis, L. A. Peletier and D. Terman, \textit{A very
singular solution of the heat equation with absorption,} Arch. Ration. Mech.
Analysis 95 (1986), 185-209.

\bibitem {CanCar}P. Cannarsa and P. Cardaliaget, \textit{H\"{o}lder estimates
in space-time for viscosity solutionsof Hamilton-Jacobi equations,} Comm. Pure
Appl. Math. 63 (2010) 590--629.

\bibitem {CLS}M. Crandall, P.\ Lions and P. Souganidis, \textit{Maximal
solutions and universal bounds for some partial differential equations of
evolution, }Arch. Rat. Mech. Anal. 105 (1989), 163-190.

\bibitem {GiGuKe}B. Gilding, M. Guedda and R. Kersner, \textit{The Cauchy
problem for }$u_{t}=\Delta u+\left\vert \nabla u\right\vert ^{q},$ J. Math.
Anal. Appl. 284 (2003), 733-755.

\bibitem {KP}S. Kamin and L. A. Peletier, \textit{Singular solutions of the
heat equation with absorption,} Proc. Amer. Math. Soc. 95 (1985), 205-210.

\bibitem {KP2}S. Kamin \& L.A. Peletier, \textit{Source-type solutions of
degenerate diffusion equations with absorption, Israel J. Math.,} 50 (1985), 219-230.

\bibitem {KVa}S. Kamin and J. L. Vazquez, \textit{Singular solutions of some
nonlinear parabolic equations,} J. Analyse Math. 59 (1992), 51-74.

\bibitem {KPVa}S. Kamin, L. A. Peletier and J. L. Vazquez,\textit{
Classification of singular solutions of a nonlinear heat equation,} Duke Math.
J. 58 (1989), 601-615.

\bibitem {KVe}S. Kamin and L. V\'{e}ron, \textit{Existence and uniqueness of
the very singular solution of the porous media equation with absorption,} J.
Analyse Math. 51 (1988), 245--258.

\bibitem {Kw}M. Kwak, \textit{A porous media equation with absorption. II.
Uniqueness of the very singular solution,} J. Math. Anal. Appl., 223 (1998), 111-125.

\bibitem {Le1}G. Leoni, \textit{A very singular solution for the porous media
equation }$u_{t}=\Delta(u^{m})-u^{p}$\textit{ when} $0<m<1$\textit{,} J. Diff.
Equ. 132 (1996), 353--376.

\bibitem {Li}G. Lieberman, Second order parabolic differential equations,
World Scientific Publishing Co. Pte. Ltd. (1996).

\bibitem {Le2}G. Leoni, \textit{On very singular self-similar solutions for
the porous media equation with absorption,} Diff. Int. Equ. 10 (1997), 1123--1140.

\bibitem {LSU}O.A. Ladyzenskaja, V.A. Solonnikov and N.N. Ural'Ceva, Linear
and Quasilinear Equations of Parabolic Type, Transl. Math. Monogr. 23, Amer.
Math. Soc., Providence, 1968.

\bibitem {MaVe}M. Marcus and L.V\'{e}ron, \textit{Initial trace of positive
solutions of some nonlinear parabolic equation,} Comm. Part. Diff. Equ., 24
(1999), 1445-1499.

\bibitem {MouV}I. Moutoussamy and L. V\'{e}ron, \textit{Source type positive
solutions of nonlinear parabolic inequalities,} Ann. Normale Sup. Di Pisa, 4
(1989), 527-555.

\bibitem {Os}L. Oswald, \textit{Isolated positive singularities for a
nonlinear heat equation}, Houston J. Math.14 (1988), 543--572.

\bibitem {Pri}A. Prignet, \textit{Existence and uniqueness of "entropy"
solutions of parabolic prolems with }$L^{1}$\textit{ data,} Nonlinear Anal. 28
(1997), 1943-1954.

\bibitem {PT}L. A. Peletier and D. Terman, \textit{A very singular solution of
the porous media equation with absorption,} Journal of Differential Equations.
65 (1986), 396-410.

\bibitem {PW}L. A. Peletier and J. Wang, \textit{A very singular solution of a
quasilinear degenerate diffusion equation with absorption,} Trans. Amer. Math.
Soc., 307 (1988), 813-826.

\bibitem {PZJ}A. Peletier and J.\ N.\ Zhao, \textit{Source-type solutions of
the porous media equation with absorption: the fast diffusion case,} Nonlinear
Anal. 14 (1990), 107--121.

\bibitem {Po}A. Porretta, \textit{Existence results for nonlinear parabolic
equations via strong convergence of trucations, }Ann. Mat. Pura Appl., 177
(1999), 143-172.

\bibitem {QW}Y. Qi and M. Wang, \textit{The self-similar profiles of
generalized KPZ equation,} Pacific J. Math. 201 (2001), 223-240.

\bibitem {Tr}N. Trudinger, \textit{Pointwise estimates and quasilinear
parabolic equations,} Comm. Part. Diff Equ., 21 (1968), 205-226
\end{thebibliography}
\end{document}